\documentclass[a4paper,10pt]{article}
\title{Expected number and distribution of critical points of  real Lefschetz pencils}
\author{Michele Ancona \thanks{Institut Camille Jordan, Umr Cnrs 5208, Universit\'{e} Claude Bernard Lyon 1. ancona@math.univ-lyon1.fr} }
\usepackage[a4paper,bindingoffset=0.2in,%
            left=1in,right=1in,top=1in,bottom=1in,%
            footskip=.25in]{geometry}
\usepackage[latin1]{inputenc}
\usepackage[T1]{fontenc}
\usepackage[english]{babel}
\usepackage{amsmath,amssymb,amsthm,amscd}
\usepackage{indentfirst}
\usepackage{fancyhdr}

\usepackage{young}
\usepackage{graphicx}

\usepackage[vcentermath]{youngtab}

\theoremstyle{plain}
\newtheorem{theo}{Theorem}[section]
\newtheorem{lemm}[theo]{Lemma}
\newtheorem{prop}[theo]{Proposition}

\newtheorem{cor}[theo]{Corollary}
\newtheorem{oss}[theo]{Remark}

\theoremstyle{definition}
\newtheorem{defi}[theo]{Definition}

\RequirePackage[colorlinks,linkcolor=blue,citecolor=blue,urlcolor=blue]{hyperref}
\usepackage{enumerate}
\usepackage{color}
\usepackage{pst-fill,pst-grad,pst-plot,pst-eucl,pstricks-add,pst-node}

\newcommand{\dV}{\textrm{dV}}
\newcommand{\Fix}{\textrm{Fix}}
\newcommand{\valcrit}{\textrm{valcrit}}
\newcommand{\Sym}{\textrm{Sym}}

\newcommand{\norm}[1]{\left\lVert#1\right\rVert}

\renewcommand{\Im}{{\rm Im}\,}

\allowdisplaybreaks

\begin{document}
\maketitle

\section*{Abstract}

We give an asymptotic probabilistic real Riemann-Hurwitz formula computing the expected real ramification index of a random covering over the Riemann sphere.
More generally,  we study the asymptotic expected number and distribution of critical points of a random real Lefschetz pencil over a smooth real algebraic variety.

\section{Introduction}
The Riemann-Hurwitz formula says that the total ramification index  of a degree $d$ branched covering  $f:\Sigma\rightarrow \Sigma'$  between two compact Riemann surfaces equals $d\cdot\chi(\Sigma')-\chi(\Sigma)$. In particular, if $\Sigma'=\mathbb{C} \mathbb{P}^1$, the total ramification index is  $2d+2g-2$, where $g$ is the genus of $\Sigma$.\\
More generally, if $u:X\dashrightarrow \mathbb{C} \mathbb{P}^1$ is a Lefschetz pencil on a complex manifold $X$ of dimension $n$, then 
$$(-1)^n\#\textrm{crit}(u)=\chi(X)-2\chi(F)+\chi(Y)$$ 
where $F$ is a smooth fiber of $u$ and $Y$ is the base locus of $u$.

The questions that motivate this paper are the following: \emph{how do  these critical points distribute on the variety? When $u$ is defined over $\mathbb{R}$, what about the number  of real critical points?}\\
We answer these questions by computing the asymptotic expected number of real critical points of real Lefschetz pencils and also the asymptotic distribution of such points.

The chosen random setting has already been considered by  Shiffman and Zelditch in \cite{sz} to study the integration current over the  zero locus of a  random global section of a line bundle over a complex projective manifold.

In the real case, Kac \cite{kac}, Kostlan \cite{ko} and Shub and Smale \cite{ss} computed the expected number of real roots of a random real polynomial.
In higher dimensions, Podkorytov \cite{pod} and B\"urgisser \cite{bur} computed the expected Euler characteristic of  random real algebraic submanifolds and Letendre  \cite{let} the expected volume (see \cite{lemn}  for the expected length of a random lemniscate). In \cite{gw1,gw2,gw3} Gayet and Welschinger estimated from above and below the Betti numbers of the real locus of real algebraic submanifolds  (see also \cite{ll}).
 For intersection of real quadrics, a precise asymptotic of the total Betti number has been given by Lerario and Lundberg in \cite{ll2}.
In \cite{nic} Nicolaescu computed the expected number of critical of a random smooth function on a Riemannian manifold have and how they distruibute.

\subsection*{Statements of the results}

Let $X$ be a smooth real projective manifold of dimension $n$, that is a complex projective manifold equipped with an anti-holomorphic involution $c_X$, called the real structure. We denote by $\mathbb{R} X=\Fix(c_X)$  its real locus. Let $\mathcal{L}$ be a positive real line bundle over $X$.
For large $d$, for almost all pairs  $(\alpha,\beta)\in  H^0(X;\mathcal{L}^d)^2 $ 
(resp. $\mathbb{R} H^0(X;\mathcal{L}^d)^2)$ of (real) global section, the map $u_{\alpha\beta}:X\dashrightarrow \mathbb{C} \mathbb{P}^1$ defined by $x\mapsto [\alpha(x):\beta(x)]$ is a (real) Lefschetz pencil, see Proposition \ref{lefpen}. Recall that a real Lefschetz pencil is a Lefschetz pencil $u:X\dashrightarrow \mathbb{C} \mathbb{P}^1$ such that $\textrm{conj}\circ u= u\circ c_X.$
\begin{defi} We denote the set of critical points of $u_{\alpha\beta}$ by $\textrm{crit}(u_{\alpha\beta})$ and by $\mathbb{R} \textrm{crit}(u_{\alpha\beta})=\textrm{crit}(u_{\alpha\beta})\cap \mathbb{R} X$ the set of real critical points. 
\end{defi} 
The number  of real critical points of a real Lefschetz pencil depends on the pair $(\alpha, \beta)$. The main theorem of this paper is the computation of the expected value of this number. Recall that, by definition, the expected value of $\#\mathbb{R} \textrm{crit}(u_{\alpha\beta})$ equals
$$\mathbb{E}[\#\mathbb{R} \textrm{crit}(u_{\alpha\beta})]=\int_{(\alpha,\beta)\in \mathbb{R} H^0(X,\mathcal{L}^d)^2}(\#\mathbb{R} \textrm{crit}(u_{\alpha\beta}))\textrm{d}\mu(\alpha,\beta).$$
\begin{theo}\label{num}
Let $X$ be a smooth real projective manifold of dimension $n$ and $(\mathcal{L},h)$ be a  real Hermitian line bundle over $X$ with positive curvature. Then 

$$
\lim_{d\rightarrow +\infty}\frac{1}{\sqrt{d}^n}\mathbb{E}[\#\mathbb{R} \textrm{crit}(u_{\alpha\beta})]=\left\{\begin{array}{ll}
\frac{n!!}{(n-1)!!}e_{\mathbb{R}}(n)\frac{\pi}{2}\textrm{Vol}_h(\mathbb{R} X) &\textrm{if n is odd}
\\\frac{n!!}{(n-1)!!}e_{\mathbb{R}}(n)\textrm{Vol}_h(\mathbb{R} X)& \textrm{if n is even}.
\end{array}\right.$$
\end{theo}

In this theorem, $\textrm{Vol}_h(\mathbb{R} X)$ is the volume of $\mathbb{R} X$ with respect to the Riemannian volume form  $\dV_h$ induced by the positive curvature  of the  metric $h$. The  probability measure we consider is a natural Gaussian probability on $\mathbb{R} H^0(X;\mathcal{L}^d)^2$ (see Section 2.1) and $e_{\mathbb{R}}(n)$ is the expected value of (the absolute value) of the determinant of real symmetric matrices (for the explicit values of $e_{\mathbb{R}}(n)$, see \cite[Section 2]{gw2}).\\
We recall that $e_{\mathbb{R}}(1)=\sqrt{\frac{2}{\pi}}$, then we have:
\begin{cor} Let $(\Sigma,c_{\Sigma})$ be a real Riemann surface and $(\mathcal{L},h)$ be a real Hermitian line bundle of degree $1$. Then, for every pair $(\alpha,\beta)\in\mathbb{R} H^0(X;\mathcal{L}^d)^2$ without common zeros, the map $u_{\alpha\beta}$ is a degree $d$ branched covering between $\Sigma$ and $\mathbb{C} \mathbb{P}^1$ and  the expected real total ramification index of $u_{\alpha\beta}$  is equivalent to $$\sqrt{\frac{\pi}{2}}Vol_h(\mathbb{R}\Sigma)\sqrt{d}$$ as $d$  tends to $+\infty$.\\
\end{cor}

Theorem \ref{num} is a consequence of a more precise equidistribution   result. In order to introduce it, let us define a natural empirical measure associated with the real  critical points of a Lefschetz pencil as follows.
For any pair $(\alpha,\beta)\in  \mathbb{R} H^0(X;\mathcal{L}^d)^2$ of real global sections of $\mathcal{L}^d$, we define $$  \mathbb{R} \nu_{\alpha\beta}=\sum_{x\in \mathbb{R} \textrm{crit}(u_{\alpha\beta})} \delta_x.$$ 

\begin{theo}\label{eqreal}
Let $X$  be a smooth real projective manifold of dimension $n$ and $(\mathcal{L},h)$ be  a  real Hermitian line bundle over $X$ with positive curvature $\omega$. Then 

$$
\lim_{d\rightarrow +\infty}\frac{1}{\sqrt{d}^n}\mathbb{E}[\mathbb{R} \nu_{\alpha\beta}]=\left\{\begin{array}{ll}
\frac{n!!}{(n-1)!!}e_{\mathbb{R}}(n)\frac{\pi}{2}\dV_h &\textrm{if n is odd}
\\\frac{n!!}{(n-1)!!}e_{\mathbb{R}}(n)\dV_h& \textrm{if n is even}.
\end{array}\right.$$
weakly in the sense of distributions. Here, $\dV_h$ is the Riemannian volume form induced by the curvature $\omega$.
\end{theo}

Theorem \ref{eqreal} says that, for any continuous function  $\varphi\in C^0(\mathbb{R} X)$, we have 
$$
\lim_{d\rightarrow +\infty}\frac{1}{\sqrt{d}^n}\mathbb{E}[\mathbb{R} \nu_{\alpha\beta}](\varphi)=\left\{\begin{array}{ll}
\frac{n!!}{(n-1)!!}e_{\mathbb{R}}(n)\frac{\pi}{2}\int_{\mathbb{R} X}\varphi \dV_h &\textrm{if n is odd}
\\\frac{n!!}{(n-1)!!}e_{\mathbb{R}}(n)\int_{\mathbb{R} X}\varphi \dV_h & \textrm{if n is even}.
\end{array}\right.$$
where the expected value is defined by $$\mathbb{E}[\mathbb{R} \nu_{\alpha\beta}](\varphi)=\int_{\mathbb{R} H^0(X;\mathcal{L}^d)^2}\sum_{x\in \mathbb{R} \textrm{crit}(u_{\alpha\beta})}\varphi(x)\textrm{d}\mu(\alpha,\beta).$$

In the complex case, we obtain a similar  equidistribution theorem, whose proof   follows along the same lines. For any pair $(\alpha,\beta)\in  H^0(X;\mathcal{L}^d)^2 $ 
 of  global sections of $\mathcal{L}^d$, we define $$ \nu_{\alpha\beta}=\displaystyle \sum_{x\in \textrm{crit}(u_{\alpha\beta})} \delta_x $$ to be the empirical measure associated with the   critical points of the pencil $u_{\alpha\beta}$.

\begin{theo}\label{eqcom}
Let $X$ be a smooth complex projective manifold of dimension $n$ and $(\mathcal{L},h)$ be  a Hermitian line bundle over $X$ with positive curvature $\omega$. Then 

$$\lim_{d\rightarrow +\infty}\frac{1}{d^n}\mathbb{E}[\nu_{\alpha\beta}]=(n+1)\omega^n$$
weakly in the sens of distribution.
\end{theo}
As before, Theorem \ref{eqcom} says that, for any continuous function  $\varphi$ on $X$, we have

$$\lim_{d\rightarrow +\infty}\frac{1}{d^n}\mathbb{E}[\nu_{\alpha\beta}](\varphi)=(n+1)\cdot\int_X\varphi \omega^n.$$
\subsubsection*{Organisation of the paper}
In Section \ref{not}, we introduce the Gaussian measure on $H^0(X;\mathcal{L}^d)^2$ associated with a Hermitian  line  bundle $(\mathcal{L},h)$ over a complex manifold $X$. We also give the same construction for the real case. We follow the approach of  \cite{sz,gw1,gw2}. In Section \ref{secpe}, we present some classical results about Lefschetz pencils on complex manifolds.
In Sections \ref{lefpeak} and \ref{lefincvar} we introduce our main tools, namely  the H\"ormander peak sections (see also \cite{gw2}, \cite{tian}) and the incidence manifold (see \cite{ss}).
Section \ref{chproof}  is completely devoted to the proofs of the Theorems  \ref{num}, \ref{eqreal} and \ref{eqcom}. In Sections \ref{lefcoarea} and \ref{lefcomputation}, we prove the equidistribution of critical points of a (real) Lefschetz pencil over a (real) algebraic variety $X$. This will be done using coarea formula and peak sections. These ideas are taken from \cite{gw2}. In Section \ref{seccomput} we will compute the universal constant by direct computation.

\subsubsection*{Acknowledgments}
I am very grateful to my advisor Jean-Yves Welschinger for all the time he devoted to me and for all the fruitful discussions we had. 

 This work was performed within the framework of the LABEX MILYON (ANR-10-LABX-0070)
of Universit\'e de Lyon, within the program "Investissements d'Avenir"
(ANR-11-IDEX-0007) operated by the French National Research Agency (ANR).

\section{Definitions and main tools}
\subsection{Notations}\label{not}

Let $X$ be a complex manifold of dimension $n$.
Let $\mathcal{L}\rightarrow X$ be a holomorphic line bundle equipped with a Hermitian metric $h$ of positive curvature $\omega\in\Omega^{(1,1)}(X,\mathbb{R})$. The curvature form induces a K\"ahler metric   and a normalized volume form  $\textrm{d}x=\frac{\omega^n}{\int_X\omega^n}$ on $X$.\\
The Hermitian metric $h$ induces a Hermitian metric $h^d$ on $\mathcal{L}^d$ for any integer $d>0$ and also a $L^2$-Hermitian product $\langle\cdot,\cdot \rangle_{L^2}$ on the space $H^0(X;\mathcal{L}^d)$ of global holomorphic sections of $\mathcal{L}^d$. It is defined by 
$$\langle\alpha,\beta\rangle_{L^2}=\int_Xh^d(\alpha,\beta)\textrm{d}x$$
for any $\alpha,\beta$ in $H^0(X;\mathcal{L}^d)$.\\
This $L^2$-Hermitian product induces a Gaussian measure  on $H^0(X;\mathcal{L}^d)^2$
defined by

$$\mu(A)=\frac{1}{\pi^{2N_d}}\int_Ae^{-\norm{\alpha}_{L^2}^2-\norm{\beta}_{L^2}^2}\textrm{d}\alpha\textrm{d}\beta$$
for any open subset $A\subset H^0(X;\mathcal{L}^d)^2$ where $\textrm{d}\alpha\textrm{d}\beta$ is the Lebesgue measure associated with $\langle\cdot,\cdot\rangle_L^2$ and $N_d=\dim_{\mathbb{C}}H^0(X;\mathcal{L}^d)$.\\
Finally, a Lefschetz pencil on $X$ is a rational map $u:X\dashrightarrow\mathbb{C}\mathbb{P}^1$ having only non degenerated critical points and defined by two global sections of a holomorphic line bundle with smooth and transverse vanishing loci.\\ 

All these definitions have a real counterpart. 
\begin{itemize}
\item Let $X$ be a real algebraic variety of dimension $n$, that is a complex manifold equipped with an anti-holomorphic involution $c_X$. We denote by
$\mathbb{R}X=\Fix(c_X)$ its real locus. 
\item A real holomorphic line bundle $u:\mathcal{L}\rightarrow X$  is a line bundle  equipped with an anti-holomorphic involution $c_{\mathcal{L}}$ such that $p\circ c_X=c_{\mathcal{L}}\circ p$ and $c_{\mathcal{L}}$ is complex-antilinear in the fibers.
\item We denote by $\mathbb{R} H^0(X;\mathcal{L})$ the real vector space of real global section of $\mathcal{L}$, i.e. sections $s\in H^0(X;\mathcal{L})$ such that $s \circ c_X=c_{\mathcal{L}}\circ s$.
\item A real Hermitian metric on $\mathcal{L}$ is a Hermitian metric $h$ such that $c_{\mathcal{L}}^*h=\bar{h}$. If $(\mathcal{L},h)$ is a   line bundle over $X$ with positive curvature $\omega$, then $\omega(.,i.)$ is a Hermitian metric over $X$ and its real part defines a Riemannian metric over $\mathbb{R} X$. We denote  the Riemannian volume form induced by this metric by $\dV_h$.
\item The $L^2$-Hermitian product $\langle\cdot,\cdot\rangle_{L^2}$ on $H^0(X;\mathcal{L}^d)$ restricts to a $L^2$-scalar product on $\mathbb{R} H^0(X;\mathcal{L}^d)$, also denoted  by $\langle\cdot,\cdot\rangle_{L^2}$. Then, as in the complex case, also in the real case we have a natural Gaussian measure on $\mathbb{R} H^0(X;\mathcal{L}^d)^2$ defined by

$$\mu(A)=\frac{1}{\pi^{N_d}}\int_Ae^{-\norm{\alpha}_{L^2}^2-\norm{\beta}_{L^2}^2}\textrm{d}\alpha\textrm{d}\beta$$
for any open subset $A\subset \mathbb{R} H^0(X;\mathcal{L}^d)^2$ where $\textrm{d}\alpha\textrm{d}\beta$ is the Lebesgue measure associated with $\langle\cdot,\cdot\rangle_{L^2}$ and $N_d=\dim_{\mathbb{C}}H^0(X;\mathcal{L}^d)=\dim_{\mathbb{R}}\mathbb{R} H^0(X;\mathcal{L}^d)$.\\
\item A real Lefschetz pencil $u:X\dashrightarrow\mathbb{C} \mathbb{P}^1$ is a Lefschetz pencil  such that $p\circ c_X=\textrm{conj}\textrm{conj}\circ p$.
\end{itemize}
We conclude this section by introducing some notation on symmetric matrices.
 \begin{defi}\label{mat} For any $n\in \mathbb{N}^*$, we denote by $\Sym(n,\mathbb{R})$ the real vector space of real symmetric matrices of size $n\times n$. It is a vector space of dimension $\frac{n(n+1)}{2}$. We equip it with the basis $\mathcal{B}$ given by $\tilde{E}_{jj}$ and $\tilde{E}_{ij}=E_{ij}+E_{ji}$ for $1\leqslant i<j\leqslant n$, where, for any $k,l$ with $1\leqslant k,l\leqslant n$, we denote by $E_{kl}$  the elementary matrix whose entry at the $i$-th row and $j$-th column equals $1$ if $(i,j)=(k,l)$ and $0$ otherwise.
 \end{defi}
We equip $\Sym(n,\mathbb{R})$ with the scalar product turning $\mathcal{B}$ into an orthonormal basis. Let $\mu_{\mathbb{R}}$  the associated Gaussian probability measure. We then set 
$$e_{\mathbb{R}}(n)=\int_{A\in \Sym(n,\mathbb{R})}|\det A| \textrm{d}\mu_{\mathbb{R}}(A).$$

\subsection{Lefschetz pencils}\label{secpe}

In this section, we compute the asymptotic value of the number of critical points of a Lefschetz pencil (see also \cite[Section 1]{gw1}).\\
Recall that a \emph{Lefschetz fibration} is a map $X\rightarrow \mathbb{C} \mathbb{P}^1$ with only non degenerate critical points. The following proposition is a kind of Riemann-Hurwitz formula for Lefschetz pencils, for a proof see \cite[Proposition 1]{gw1}.

\begin{prop}\label{hurlef} Let $X$ be a smooth complex projective manifold of positive dimension $n$ equipped with a Lefschetz fibration $p:X\rightarrow\mathbb{C} \mathbb{P}^1$ and let $F$ be a regular fiber of $p$. Then we have the following equality:
$$\chi(X)=2\chi(F)+(-1)^n\#\textrm{crit}(p).$$
\end{prop}
Remark that if $u:X\dashrightarrow\mathbb{C} \mathbb{P}^1$ a Lefschetz pencil and we blow-up the base locus $\textrm{Base}(u)\doteqdot Y$, then we obtain a Lefschetz fibration $\tilde{u}:\tilde{X}\doteqdot Bl_YX\rightarrow\mathbb{C} \mathbb{P}^1$. By additivity of the Euler characteristic, we have that $\chi(\tilde{X})=\chi(X)+\chi(Y)$, then by Proposition \ref {hurlef} we have \begin{equation}
\chi(X)=2\chi(F)-\chi(Y)+(-1)^n\#\textrm{crit}(u).
\end{equation}

\begin{prop}\label{asymlef}
Let $\mathcal{L}$ be an ample line bundle over a complex manifold $X$ of dimension $n$. For almost all global sections $(\alpha, \beta)^2\in H^0(X;\mathcal{L}^d)^2$, the map $u_{\alpha\beta}$ defined by $x\mapsto [\alpha(x):\beta(x)]$ is a Lefschetz pencil (see Prop. \ref{lefpen}).
Then, as $d$ goes to infinity, we have

\begin{equation}
\#\textrm{crit}(u_{\alpha\beta})=(n+1)\left(\int_Xc_1(\mathcal{L})^n\right)d^n+O(d^{n-1}).
\end{equation}

\end{prop}
\begin{proof}  We will follow the lines of Lemma 2, Lemma 3 and Proposition 4 of \cite{gw1}.\\
We have $\chi(F)=\int_F c_{n-1}(F)$ and $\chi(Y)=\int_Y c_{n-2}(Y)$.
We remark that the base locus is the intersection of the zero locus of $\alpha$ and $\beta$, that is $Y=Z_{\alpha}\cap Z_{\beta}$. 
A regular fiber $F$ over $[a,b]\in\mathbb{C} \mathbb{P}^1$ is the zero locus of the section $b\alpha-a\beta\in H^0(X;\mathcal{L}^d)$, thus the normal bundle $N_{X/F}$ is $\mathcal{L}^d_{\mid F}$ 
To compute $\chi(F)$ we will use the adjunction formula.
We have 
$$0\rightarrow TF\rightarrow TX_{\mid F} \rightarrow N_{X/F}\rightarrow 0$$ 
then we have $c(X)_{\mid F}=c(F)\wedge c(\mathcal{L}^d)_{\mid F}$, that is $$\big(1+c_1(X)+\dots+c_n(X)\big)_{\mid F}=(1+c_1(F)+\dots+c_{n-1}(F))\wedge(1+dc_1(\mathcal{L})).$$
If we develop this we have $c_1(X)=c_1(F)+dc_1(\mathcal{L})$  and, for $j\in \{2,\dots,n-1\}$, we have $c_j(X)_{\mid F}=c_j(F)+dc_1(\mathcal{L})_{\mid F}\wedge c_{j-1}(F).$
Then, summing up the term,

$$c_j(F)=\sum_{k=0}^j(-1)^kd^kc_1(\mathcal{L})^k_{\mid F}\wedge c_{j-k}(X)_{\mid F}.$$
In particular, for $j=n-1$ we have

$$c_{n-1}(F)=\sum_{k=0}^{n-1}(-1)^kd^kc_1(\mathcal{L})^k_{\mid F}\wedge c_{n-k-1}(X)_{\mid F}.$$
Then $\chi(F)$ is equal to $\int_F\sum_{k=0}^{n-1}(-1)^kd^kc_1(\mathcal{L})^k)_{\mid F}\wedge c_{n-k-1}(X)_{\mid F}$.
But, for $\alpha\in H_{dR}^{2n-2}(X)$, we have that 

$$\int_F\alpha_{\mid F}=\int_X\alpha\wedge c_1(\mathcal{L}^d)$$
so, $$\chi(F)=\sum_{k=0}^{n-1}\int_X(-1)^kd^{k+1}c_1(\mathcal{L})^{k+1}\wedge c_{n-k-1}(X)$$
and asymptotically we get $$\chi(F)\sim(-1)^{n-1}(\int_Xc_1(\mathcal{L})^n)d^n.$$
For $Y=Z_{\alpha}\cap Z_{\beta}$, the same argument gives us

$$c_j(Y)=\sum_{k=0}^j(-1)^kd^kc_1(\mathcal{L})^k_{\mid Y}\wedge c_{j-k}(Z_{\alpha})_{\mid Y}.$$
But, as before, $$c_{j-k}(Z_{\alpha})=\sum_{h=0}^{j-k}(-1)^hd^hc_1(\mathcal{L})^h\wedge c_{j-k-h}(X).$$
and so,  replacing in the above equation 
$$c_j(Y)=\sum_{k=0}^j(-1)^kd^kc_1(\mathcal{L})^k_{\mid Y}\wedge (\sum_{h=0}^{j-k}(-1)^hd^hc_1(\mathcal{L})^h_{\mid Y}\wedge c_{j-k-h}(X)_{\mid Y})$$
For $j=n-2$ we have
$$c_{n-2}(Y)=
\sum_{k=0}^{n-2}(-1)^kd^kc_1(\mathcal{L})^k_{\mid Y}\wedge (\sum_{h=0}^{n-2-k}(-1)^hd^hc_1(\mathcal{L})^h_{Y}\wedge c_{n-2-k-h}(X)_{\mid Y}) $$
and this is equivalent to 
$$\sum_{k=0}^{n-2}(-1)^{n-2}d^{n-2}c_1(\mathcal{L})^{n-2}_{\mid Y}
=
(-1)^{n-2}(n-1)d^{n-2}c_1(\mathcal{L})^{n-2}_{\mid Y}$$
 as $d\rightarrow\infty$.
So we have, as $d\rightarrow\infty$,
$$\chi(Y)\sim
(-1)^{n-2}(n-1)d^{n-2}\int_Yc_1(\mathcal{L})^{n-2}_{\mid Y}=$$
$$=(-1)^{n-2}(n-1)d^{n-1}\int_{Z_{\alpha}}c_1(\mathcal{L})^{n-2}\wedge c_1(\mathcal{L})=$$
$$=(-1)^{n-2}(n-1)(\int_Xc_1(\mathcal{L})^n)d^n.$$
Combining this with $\chi(X)=2\chi(F)-\chi(Y)+(-1)^n\#\textrm{crit}(u_{\alpha\beta})$  we obtain the result.
\end{proof}

\subsection{H\"ormander's peak sections}\label{lefpeak}
In this section we recall the theory of H\"ormander's peak sections, an essential tool for our proofs of Theorems   \ref{eqreal} and \ref{eqcom} (see also \cite{hor}, \cite{tian}, \cite{gw2}).
Let $\mathcal{L}$ be a holomorphic line bundle over a smooth complex projective manifold equipped with a Hermitian metric $h$ of positive curvature $\omega$ and let $\textrm{d}x=\frac{\omega^n}{\int_X\omega^n}$ be the normalized volume form. Let $x$ be a point of $X$.
There exists, in the neighborhood of $x$, a holomorphic trivialization $e$ of $\mathcal{L}$ such that the associated potential reaches a local minimum at $x$ with Hessian of type $(1,1)$.
The following result was proved in  \cite{tian} (see also \cite{gw2}) .
\begin{lemm}\label{peak}
Let $(\mathcal{L},h)$ be a holomorphic Hermitian line bundle of positive curvature $\omega$ over a smooth complex projective manifold $X$. Let $x\in X$, $(p_1,\dots,p_n)\in\mathbb{N}^n$ and $p'>p_1+\dots+p_n$. There exists $d_0\in\mathbb{N}$ such that for every $d>d_0$, the bundle $\mathcal{L}^d$ has a global holomorphic section $\sigma$ satisfying $\int_Xh^d(\sigma,\sigma)\textrm{d}x=1$ and 
$$\int_{X\setminus B(x,\frac{\log d}{\sqrt{d}})}h^d(\sigma,\sigma)\textrm{d}x=O(\frac{1}{d^{2p'}})$$

Moreover, if $(x_1,\dots,x_n)$ are local holomorphic coordinates in the neighborhood of $x$, we can assume that in a neighborhood of $x$,
$$\sigma(x_1,\dots,x_n)=\lambda(x_1^{p_1}\cdots x_n^{p_n}+O(\norm{x}^{2p'}))e^d\big(1+O(\frac{1}{d^{2p'}})\big) $$
where 
$$\lambda^{-2}=\int_{B(x,\frac{\log d}{\sqrt{d}})}|x_1^{p_1}\cdots x_n^{p_n}|^2 h^d(e^d,e^d)\textrm{d}x$$
and $e$ is a holomorphic trivialization of $\mathcal{L}$ in the neighborhood of $x$  whose potential $\phi=-\log h(e,e)$ reaches a local minimum at $x$ with Hessian $\pi\omega(.,i.)$.
\end{lemm}

This lemma is true also in the real setting, in the following sense:
\begin{lemm}\label{realpeak}
Let $(\mathcal{L},h)$ be a real holomorphic Hermitian line bundle of positive curvature $\omega$ over a smooth real projective manifold $X$. Let $x\in \mathbb{R}X$, $(p_1,\dots,p_n)\in\mathbb{N}^n$ and $p'>p_1+\dots+p_n$. There exists $d_0\in\mathbb{N}$ such that for every $d>d_0$, the bundle $\mathcal{L}^d$ has a global holomorphic section $\sigma$ satisfying $\int_Xh^d(\sigma,\sigma)\textrm{d}x=1$ and 
$$\int_{X\setminus B(x,\frac{\log d}{\sqrt{d}})}h^d(\sigma,\sigma)\dV_h=O(\frac{1}{d^{2p'}})$$
Moreover, if $(x_1,\dots,x_n)$ are local real holomorphic coordinates in the neighborhood of $x$, we can assume that in a neighborhood of $x$,
$$\sigma(x_1,\dots,x_n)=\lambda(x_1^{p_1}\cdots x_n^{p_n}+O(\norm{x}^{2p'}))e^d\big(1+O(\frac{1}{d^{2p'}})\big) $$
where 
$$\lambda^{-2}=\int_{B(x,\frac{\log d}{\sqrt{d}})}| x_1^{p_1}\cdots x_n^{p_n}|^2 h^d(e^d,e^d)\textrm{d}x$$
and $e$ is a real trivialization of $\mathcal{L}$ in the neighborhood of $x$  whose potential $\phi=-\log h(e,e)$ reaches a local minimum at $x$ with Hessian $\pi\omega(.,i.)$.

\end{lemm}
This real counterpart follows from Lemma \ref{peak} by averaging the peak sections with the real structure.\\
Let $\sigma_0$ be  the section given by the Lemma \ref{realpeak} with $p'=3$ 
and $p_i=0$ for all $i$, $\sigma_i$ the section given by Lemma \ref{realpeak} with $p'=3$ and $p_j=\delta_{ij}$, $\sigma_{ij}$ the section given by the Lemma \ref{realpeak} with $p_i=p_j=1$ and $p_k=0$ otherwise and $\sigma_{kk}$ the section given by the Lemma \ref{realpeak} with $p_k=2$ and $p_l$ for $l\neq k$. \\
These sections are called \emph{peak sections}. Their Taylor expansions are:

$$\sigma_0(y)=\big(\lambda_0+O(\norm{ y}^6)\big)e^d\big(1+O(\frac{1}{d^6})\big);$$
$$\sigma_i(y)=\big(\lambda_iy_i+O(\norm{ y}^6)\big)e^d\big(1+O(\frac{1}{d^6})\big) \hspace{3 mm} \forall i;$$
$$\sigma_{ij}(y)=\big(\lambda_{ij}y_iy_j+O(\norm{ y}^6)\big)e^d\big(1+O(\frac{1}{d^6})\big)\hspace{3 mm} \forall i\neq j;$$
$$\sigma_{kk}(y)=\big(\lambda_{kk}y_k^2+O(\norm{ y}^6)\big)e^d\big(1+O(\frac{1}{d^6})\big) \hspace{3 mm} \forall k.$$
The following lemma shows the asymptotic of the constants 
$\lambda_0$, $\lambda_i$, $\lambda_{ij}$ et $\lambda_{kk}$.

\begin{lemm}[Lemma 2.5 of \cite{gw2}]\label{limpeak}
Under the hypothesis of Lemma \ref{peak} and \ref{realpeak}, we have
$$\lim_{d\longrightarrow\infty}\frac{1}{\sqrt{d}^{n}}\lambda_0=\sqrt{\delta_{\mathcal{L}}}$$
$$\lim_{d\longrightarrow\infty}\frac{1}{\sqrt{d}^{n+1}}\lambda_i=\sqrt{\pi}\sqrt{\delta_{\mathcal{L}}}$$
$$\lim_{d\longrightarrow\infty}\frac{1}{\sqrt{d}^{n+2}}\lambda_{ij}=\pi\sqrt{\delta_{\mathcal{L}}}$$
$$\lim_{d\longrightarrow\infty}\frac{1}{\sqrt{d}^{n+2}}\lambda_{kk}=\frac{\pi}{\sqrt{2}}\sqrt{\delta_{\mathcal{L}}}$$
for the $L^2$-product induced by $\textrm{d}x=\frac{\omega^n}{\int_X\omega^n}$ where  $\delta_{\mathcal{L}}=\int_X c_1(\mathcal{L})^n$ is the degree of the line bundle $\mathcal{L}$.
\end{lemm}

Set
 $$H_{2x}=\{s\in H^0(X;\mathcal{L}^d) \mid s(x)=0, \nabla s(x)=0,\nabla^2s(x)=0\}$$

 $$\left(\textrm{resp.}\quad\mathbb{R} H_{2x}=\{s\in \mathbb{R} H^0(X;\mathcal{L}^d) \mid s(x)=0, \nabla s(x)=0,\nabla^2s(x)=0\}\right).$$

This space is formed by  sections whose $2$-jet vanishes at $x$. The sections $(\sigma_i)_{0\leq i\leq n}$ $(\sigma_{ij})_{1\leq i\leq j\leq n}$ provide a basis of a complement of $H_{2x}$. This basis is not orthonormal and its spanned subspace is not orthogonal to $H_{2x}$. However, this basis is aymptotically an orthonormal basis and its spanned subspace is asymptotically orthonormal to $H_{2x}$, in the following sense:

\begin{prop}[Lemma 3.1 of \cite{tian}]\label{ortpeak}
The section $(\sigma_i)_{0\leq i\leq n}$ and $(\sigma_{ij})_{1\leq i\leq j\leq n}$ have $L^2$-norm equal to $1$ and  and their pairwise $L^2$-scalar product are $O(\frac{1}{d})$. Likewise, their scalar products with every unitary element of $H_{2x}$ are $O(\frac{1}{d^{3/2}})$.
\end{prop}

\subsection{Incidence manifolds}\label{lefincvar}
Following \cite{ss}, we define an incidence manifold associated with the complex (resp. real) manifold $X$ and to the (real) positive line bundle $\mathcal{L}$. We will use this incidence manifold to prove  that,  for almost all pairs global sections $(\alpha, \beta)\in H^0(X;\mathcal{L}^d)^2$  (resp. $(\alpha, \beta)\in\mathbb{R} H^0(X;\mathcal{L}^d)^2$),  the map $u_{\alpha\beta}x\mapsto [\alpha(x):\beta(x)]$ defines   a Lefschetz pencil, see Proposition \ref{asymlef}.

Let $(\mathcal{L},h)$ be a (real) Hermitian line bundle with positive curvature $\omega$ over a (real) algebraic variety $X$ of dimension $n$.\\ 
\begin{defi} Let $\alpha, \beta\in H^0(X;\mathcal{L}^d)$ (resp. $\mathbb{R} H^0(X;\mathcal{L}^d)$) be two  (real) global sections such that the map $x\mapsto [\alpha(x):\beta(x)]$ is a Lefschetz pencil. We define \begin{enumerate}

\item the \emph{base locus} of a Lefschetz pencil as the points $x$ such that $\alpha(x)=\beta(x)=0$; 
\item the \emph{critical points} as the points $x\in X\setminus \textrm{Base}(u_{\alpha\beta})$ such that $(\alpha\nabla\beta-\beta\nabla\alpha)(x)=0 $ (this expression does not depend on the choice of a connection $\nabla$ on $\mathcal{L}$).
We denote by $\textrm{crit}(u_{\alpha\beta})$ the set of critical points of $(u_{\alpha\beta})$ and by $\mathbb{R} \textrm{crit}(u_{\alpha\beta})=\textrm{crit}(u_{\alpha\beta})\cap\mathbb{R} X$ the set of real critical points. 
\end{enumerate}
\end{defi}
 
We denote by $\Delta$ (resp. by $\mathbb{R}\Delta$) the set of  $(\alpha, \beta,x)\in H^0(X;\mathcal{L}^d)^2\times X$ (resp. $(\alpha, \beta,x)\in \mathbb{R} H^0(X;\mathcal{L}^d)^2\times\mathbb{R} X$) such that $\alpha(x)=\beta(x)=0$.
Set
$$\mathcal{I}=\big\{(\alpha,\beta,x)\in \big(H^0(X;\mathcal{L}^d)^2 \times X\big)\setminus\Delta\mid x\in \textrm{crit}(u_{\alpha\beta}) \big\}$$
$$\big(\mathrm{resp.} \quad \mathbb{R}\mathcal{I}=\big\{(\alpha,\beta,x)\in \big(\mathbb{R} H^0(X;\mathcal{L}^d)^2 \times \mathbb{R} X\big)\setminus\mathbb{R}\Delta\mid x\in \textrm{crit}(u_{\alpha\beta}) \big\}\big)$$

\begin{prop} Let $\mathcal{L}$ be a (real) holomorphic line bundle over a smooth complex (resp. real) projective manifold $X$. If $\mathcal{L}^d$ is $1$-ample, that is if the $1$-jet map 
$$H^0(X;\mathcal{L}^d)\times X\rightarrow J^1(\mathcal{L}^d)$$
$$(s,x)\mapsto j^1_x(s)=(s(x),\nabla s(x))$$  is surjective, then $\mathcal{I}$ (resp. $\mathbb{R}\mathcal{I}$) is a smooth manifold of complex  (resp. real) dimension $2N_d$, where $N_d=\dim H^0(X;\mathcal{L}^d)$. 
\end{prop}
\begin{proof}
We study the differential of the map $$q:\big(H^0(X;\mathcal{L}^d)^2 \times X\big)\setminus\Delta\rightarrow T^*X\otimes \mathcal{L}^{2d}$$ defined by
$$(\alpha,\beta,x)\mapsto(\alpha\nabla\beta-\beta\nabla\alpha)(x)\in T_x^*X\otimes \mathcal{L}_x^{2d}$$
defining $\mathcal{I}$. If we prove that $0$ is a regular value, then, by Implicit Function Theorem, we have the result.
Now, for $(\alpha,\beta,x)\in\mathcal{I}$ we have $$\textrm{d}_{\mid(\alpha,\beta,x)}q\cdot(\dot{\alpha},\dot{\beta},\dot{x})=\big(\dot{\alpha}\nabla\beta-\beta\nabla\dot{\alpha}
+\alpha\nabla\dot{\beta}-\dot{\beta}\nabla\alpha
+\alpha\nabla^2_{(\dot{x},.)}\beta-\beta\nabla^2_{(\dot{x},.)}\alpha
+\nabla_{\dot{x}}\alpha\nabla\beta-\nabla_{\dot{x}}\beta\nabla\alpha\big)(x).$$
For any $\eta\in T_x^*X\otimes \mathcal{L}_x^{2d}$ we have to prove that there exists $(\dot{\alpha},\dot{\beta},\dot{x})$ such that
$\textrm{d}_{\mid(\alpha,\beta,x)}q\cdot(\dot{\alpha},\dot{\beta},\dot{x})=\eta$.
As $(\alpha,\beta,x)\not\in\Delta$, we know that at least one between  $\alpha(x)$ and $\beta(x)$  is not zero.
Without loss of generality, suppose that $\alpha(x)\neq 0$, then, as $\mathcal{L}^d$ is $1$-ample, there exists $\dot{\beta}$ such that  $\dot{\beta}(x)=0$ and   $\alpha(x)\nabla\dot{\beta}(x)=\eta$, then $\textrm{d}_{\mid(\alpha,\beta,x)}q\cdot(0,\dot{\beta},0)=\eta$.
\end{proof}

If $\mathcal{L}$ is ample, then, for large $d$, the line bundle $\mathcal{L}^d$ is $1$-ample. 
Then, for large $d$, $\mathcal{I}$ (resp. $\mathbb{R}\mathcal{I}$ ) is a smooth manifold, called the \emph{incidence manifold}. 
The tangent space $T_{(\alpha,\beta,x)}\mathcal{I}$ of $\mathcal{I}$  at a point $(\alpha,\beta,x)$ (resp. $T_{(\alpha,\beta,x)}\mathbb{R}\mathcal{I}$ of $\mathbb{R}\mathcal{I}$) equals
$$
\big\{
(\dot{\alpha},\dot{\beta},\dot{x})\in H^0(X;\mathcal{L}^d)^2\times T_xX\mid 
\big(\dot{\alpha}\nabla\beta-\beta\nabla\dot{\alpha}
+\alpha\nabla\dot{\beta}-\dot{\beta}\nabla\alpha
+\alpha\nabla^2_{(\dot{x},.)}\beta-\beta\nabla^2_{(\dot{x},.)}\alpha\big)
(x)=0\big\}.$$
$$
\big(\textrm{resp.}\quad \big\{(\dot{\alpha}
,\dot{\beta},\dot{x})\in \mathbb{R} H^0(X;\mathcal{L}^d)^2\times T_x\mathbb{R} X\mid 
\big(\dot{\alpha}\nabla\beta-\beta\nabla\dot{\alpha}
+\alpha\nabla\dot{\beta}-\dot{\beta}\nabla\alpha
+\alpha\nabla^2_{(\dot{x},.)}\beta-\beta\nabla^2_{(\dot{x},.)}\alpha
\big)(x)=0\big\} 
\big).
$$
\begin{oss}
\begin{itemize}
\item In the equation defining the tangent space there is also the term 
$$(\nabla_{\dot{x}}\alpha\nabla\beta-\nabla_{\dot{x}}\beta\nabla\alpha)
(x).$$ However, it equals zero (both in the complex and real case) because on $\mathcal{I}$ and $\mathbb{R}\mathcal{I}$ we have the condition $(\alpha\nabla\beta-\beta\nabla\alpha)(x)=0$
so that $$\left(\nabla_{\dot{x}}\alpha\nabla\beta-\nabla_{\dot{x}}\beta\nabla
\alpha\right)(x)=\left((\nabla_{\dot{x}}\alpha\frac{\beta}{\alpha}-\nabla_{\dot{x}}\beta)\nabla\alpha\right)
(x)=0.$$
\item The incidence manifold comes equipped with two natural projections 

$$\pi_H:\mathcal{I}\rightarrow H^0(X;\mathcal{L}^d)^2\quad\textrm{and}\quad \pi_{X}:\mathcal{I}\rightarrow X$$

$$\big(\textrm{resp.}\quad\pi_{\mathbb{R} H}:\mathbb{R}\mathcal{I}\rightarrow \mathbb{R} H^0(X;\mathcal{L}^d)^2\quad\textrm{and}\quad\pi_{\mathbb{R} X}:\mathcal{\mathbb{R} I}\rightarrow\mathbb{R} X\big).$$
\end{itemize}
\end{oss}

\begin{prop}\label{lefpen}
Let $\mathcal{L}$ be an ample  holomorphic line bundle (resp. real holomorphic) over a smooth complex projective manifold $X$ (resp. real projective).
For large $d$ and for almost all pairs $(\alpha,\beta) \in H^0(X;\mathcal{L}^d)^2$ (resp. $\mathbb{R} H^0(X;\mathcal{L}^d)^2$), the map $$u_{\alpha\beta}:X\dashrightarrow\mathbb{C}\mathbb{P}^1$$
$$x\mapsto [\alpha(x):\beta(x)].$$
is a Lefschetz pencil (resp. real Lefschetz pencil).

\end{prop}
\begin{proof}
The critical points of the projection $\pi_H$ (resp. $\pi_{\mathbb{R} H}$) are exactly the triples $(\alpha,\beta,x)$ such that the Hessian $(\alpha\nabla^2\beta-\beta\nabla^2\alpha)(x)$ is degenerate.
By Sard's theorem $\valcrit(\pi_H)$ has zero Lebesgue, and then Gaussian, measure.
Also, for large $d$, the set $\Gamma$ composed by the pairs $(\alpha,\beta)\in H^0(X;\mathcal{L}^d)\times  H^0(X;\mathcal{L}^d) $
 such that $\{x\in X, \alpha(x)=\beta(x)=0\}$ is not smooth has zero Lebesgue and Gaussian measure (see for example \cite[Section 2.2]{let}).
Then $(\Gamma\cup \valcrit(\pi_H))$ has zero measure and its complement is exactly the set of pairs of sections defining a Lefschetz pencil.
\end{proof}

\section{Proof of  the main theorems}\label{chproof}
In this section we prove  Theorems \ref{num}, \ref{eqreal} and  \ref{eqcom}.  H\"ormander's peak sections and the coarea formula play an important role here.

\subsection{Coarea formula}\label{lefcoarea}
In this section we  use the incidence manifold defined in Section \ref{lefincvar} and the coarea formula to write the expected distribution of critical points of a (real)  Lefschetz pencil as an integral over $X$ (resp. $\mathbb{R} X$).

\begin{defi}
The \emph{normal jacobian} $\textrm{Jac}_Nf$ of a submersion $f:M\rightarrow N$ between two Riemannian manifolds is the determinant of the differential of the map restricted to the orthogonal complement of its kernel, that is $\textrm{Jac}_Nf=\textrm{Jac}(\textrm{d}f_{\mid (\ker \textrm{d}f)^{\perp}})$. Equivalently, if $\textrm{d}f_p$ is the differential of $f$ at $p$, then the normal jacobian is equal to $\sqrt{\det(\textrm{d}f_p\textrm{d}f_p^*)}$,
where $\textrm{d}f_p^*$ is the adjoint of $\textrm{d}f_p$ with respect to the scalar product on $T_pM$ and $T_{f(p)}N$.
\end{defi}

Let $X$ be a smooth complex (resp. real) projective manifold of dimension $n$ and $(\mathcal{L},h)$ be a (real) holomorphic line bundle with positive curvature $\omega$.

\begin{defi}
We define a Dirac measure for (real) critical points of a (real) Lefschetz pencil $u_{\alpha\beta}$ associated with a pair $(\alpha,\beta) \in H^0(X;\mathcal{L}^d)^2$ (resp. $(\alpha,\beta)\in\mathbb{R} H^0(X;\mathcal{L}^d)^2$) by  
$$ \nu_{\alpha\beta}=\displaystyle \sum_{x\in \textrm{crit}(u_{\alpha\beta})} \delta_x \qquad (\textrm{resp}.\quad \mathbb{R} \nu_{\alpha\beta}=\sum_{x\in \mathbb{R} \textrm{crit}(u_{\alpha\beta})} \delta_x).$$ 
\end{defi}
Let $\varphi$ be a continuous function on $\mathbb{R} X$. Then, by definition, we have

$$\mathbb{E}[\mathbb{R} \nu_{\alpha\beta}](\varphi)=\int_{ \mathbb{R} H^0(X;\mathcal{L}^d)^2}\sum_{x\in \textrm{crit}(u_{\alpha\beta})}\varphi(x)\textrm{d}\mu(\alpha,\beta)$$
where $\textrm{d}\mu$ is the Gaussian measure on $\mathbb{R} H^0(X;\mathcal{L}^d)^2$ constructed in Section \ref{not}. 
Finally, recall that we denote by $\pi_{\mathbb{R} H}$ and $\pi_{\mathbb{R} X}$ the two natural projections from $\mathbb{R} \mathcal{I}$ to $\mathbb{R} H^0(X;\mathcal{L}^d)^2$ and $\mathbb{R} X$. The projection $\pi_{\mathbb{R} H}$ is (almost everywhere) a local isomorphism and, by a slight abuse of notation, we will denote by $\pi_{\mathbb{R} H}^{-1}$ any local inverse.
\begin{prop} Following the notation of Section \ref{lefincvar}, we have \begin{equation}\label{realcoarea}\mathbb{E}[\mathbb{R}\nu_{\alpha\beta}](\varphi)=\int_{\mathbb{R} X}\varphi(x)\int_{\pi_{\mathbb{R} H}(\pi^{-1}_{\mathbb{R} X}(x))}\frac{1}{|(\pi_{\mathbb{R} H}^{-1})^*\textrm{Jac}_N(\pi_{\mathbb{R} X})|}\textrm{d}\mu_{\mid\pi_{\mathbb{R} H}\big(\pi^{-1}_{\mathbb{R} X}(x)\big)}\dV_h.
 \end{equation}
where the measure $\textrm{d}\mu_{\mid\pi_{\mathbb{R} H}\big(\pi^{-1}_{\mathbb{R} X}(x)\big)}$ is the following: first we restrict the scalar product $\langle\cdot,\cdot\rangle_{L^2}$ on $\mathbb{R} H^0(X;\mathcal{L}^d)^2$ to $\pi_{\mathbb{R} H}(\pi^{-1}_{\mathbb{R} X}(x))$, which is a codimension $n$ submanifold, then  we  consider the Riemannian measure associated with this metric, and finally we multiply it by the  factor $\frac{1}{\pi^{N_d}}e^{-\norm{\alpha}_{L^2}^2-\norm{\beta}_{L^2}^2}$, where $N_d=\dim H^0(X;\mathcal{L}^d)$.
\end{prop}
\begin{proof}
We denote by  $\pi_{\mathbb{R}  H}^*\textrm{d}\mu$  the pull-backed measure on $\mathbb{R}\mathcal{I}$, which is well defined since $\pi_{\mathbb{R}  H}$ is (almost everywhere) a local isomorphism. By definition of the pull-backed measure, 
the integral $$\mathbb{E}[\mathbb{R} \nu_{\alpha\beta}](\varphi)=\int_{ \mathbb{R} H^0(X;\mathcal{L}^d)^2}\sum_{x\in \textrm{crit}(u_{\alpha\beta})}\varphi(x)\textrm{d}\mu(\alpha,\beta)$$  which defines the expected value equals the following integral over the incidence manifold $\mathbb{R} \mathcal{I}$
$$\int_{\mathbb{R}\mathcal{I}}(\pi_{\mathbb{R} X}^*\varphi)(\alpha,\beta,x)(\pi_{\mathbb{R}  H}^*\textrm{d}\mu)(\alpha,\beta,x).$$ 
We use the coarea formula (see \cite[Lemma 3.2.3]{fed} or \cite[Theorem 1]{ss}) for the map $\pi_{\mathbb{R} X}$ and we obtain
 $$\mathbb{E}[\mathbb{R} \nu_{\alpha\beta}](\varphi)=\int_{\mathbb{R} X}\varphi(x)\int_{\pi^{-1}_{\mathbb{R} X}(x)}\frac{1}{|\textrm{Jac}_N(\pi_{\mathbb{R} X})|}(\pi_{\mathbb{R}  H}^*\textrm{d}\mu)_{\mid\pi^{-1}_{\mathbb{R} X(x)}}\dV_h$$
where the measure $(\pi_{\mathbb{R}  H}^*\textrm{d}\mu)_{\mid\pi^{-1}_X(x)}$ is the following: first we restrict the (singular) metric $\pi_H^*\langle\cdot,\cdot\rangle_{L^2}$ on $\mathbb{R}\mathcal{I}$ to $\pi^{-1}_{\mathbb{R} X}(x)$, that is a codimension $n$ submanifold, then  we  consider the Riemannian measure associated with this metric, and finally we multiply it by the  factor $\frac{1}{\pi^{N_d}}e^{-\norm{\alpha}_{L^2}^2-\norm{\beta}_{L^2}^2}$, where $N_d=\dim H^0(X;\mathcal{L}^d)$.
Another application of coarea formula for the map $\pi_{\mathbb{R} H}$ gives us the result.
  
\end{proof} 
The space $\pi_{\mathbb{R} H}(\pi^{-1}_{\mathbb{R} X}(x))$ is  formed by  pairs $(\alpha,\beta)\in \mathbb{R} H^0(X;\mathcal{L}^d)^2$  such that $x\in \mathbb{R} \textrm{crit}(u_{\alpha\beta})$. In the next section we will identify this space with an intersection of some quadrics in the vector space $\mathbb{R} H^0(X;\mathcal{L}^d)^2$.\\
In the complex case, the same argument gives us, for any continuous function $\varphi$ on $X$ 
  
   \begin{equation}\label{coarea}\mathbb{E}[\nu_{\alpha\beta}](\varphi)=\int_{X}\varphi(x)\int_{\pi_H(\pi^{-1}_{X}(x))}\frac{1}{|(\pi_H^{-1})^*\textrm{Jac}_N(\pi_{X})|}\textrm{d}\mu_{\mid\pi_H(\pi^{-1}_{X}(x))}\dV_h.
   \end{equation}

 \subsection{Computation of the normal jacobian}\label{lefcomputation}
  In this section we compute the normal jacobian that appears in (\ref{realcoarea}) and (\ref{coarea}). We follow the notations of Sections \ref{lefpeak}, \ref{lefincvar} and \ref{lefcoarea}.  The main result of this section is the following proposition:
  \begin{prop}\label{espform} Following the notation of Sections \ref{lefincvar} and \ref{lefcoarea}, under the hypothesis of Theorem \ref{eqreal}, we have:
 $$
 \mathbb{E}[\mathbb{R}\nu_{\alpha\beta}](\varphi)=
\int_{x\in \mathbb{R} X}\varphi(x)R_d(x)\dV_h,
$$ 
where $$R_d(x)=\sqrt{\pi d}^n(\int_{Q}
\frac{|\det(a_0b_{ij}-b_0a_{ij})|}{\sqrt{\det\left((a_ia_j+b_ib_j)+(a_0^2+b_0^2)Id\right)}}
\textrm{d}\mu_Q +O(\frac{1}{\sqrt{d}}))$$
and $Q\subset\mathbb{R}^{2(n+1)+n(n+1)}$ is the product of the intersection of quadrics
$$\tilde{Q}=\{(a_0,b_0,\dots,a_n,\dots,b_n)\in \mathbb{R}^{2(n+1)}\mid a_0b_i-a_ib_0=0\quad\forall i=1,\dots,n\}$$
with the vector space $\mathbb{R}^{n(n+1)}$ of coordinates $a_{ij}$ and $b_{ij}$ for $1\leq i\leq j\leq n$ and
 $$\textrm{d}\mu_Q=
\frac{e^{-\sum_ia_i^2-\sum_i
b_i^2-\sum_{i,j}a_{ij}^2-\sum_{i,j}b_{ij}^2}}{\pi^{n+1+\frac{n(n+1)}{2}}}\dV_Q$$
where $\dV_Q$ is the Riemannian volume form of  $Q$.

\end{prop}
The remaining part of this section is devoted to the proof of Proposition \ref{espform}. Our main tool will be the peak sections defined in Section  \ref{lefpeak}.
 
 We fix a point $x\in  X$ (resp. $x\in\mathbb{R} X$) and we want to compute the integral \begin{equation}\label{codensity}\int_{\pi_{\mathbb{R} H}(\pi^{-1}_{\mathbb{R} X}(x))}\frac{1}{|(\pi_{\mathbb{R} H}^{-1})^*\textrm{Jac}_N(\pi_{\mathbb{R} X})|}\textrm{d}\mu_{\mid\pi_{\mathbb{R} H}(\pi^{-1}_{\mathbb{R} X}(x))}
  \end{equation}
that appears in (\ref{realcoarea}). 
We recall that the tangent space of $\mathcal{I}$ (resp. $\mathbb{R} \mathcal{I}$) at $(\alpha,\beta,x)$ is 
$$
\big\{(\dot{\alpha},\dot{\beta},\dot{x})\in H^0(X;\mathcal{L}^d)^2\times T_xX\mid 
\big(\dot{\alpha}\nabla\beta-\beta\nabla\dot{\alpha}
+\alpha\nabla\dot{\beta}-\dot{\beta}\nabla\alpha
+\alpha\nabla^2_{(\dot{x},.)}\beta-\beta\nabla^2_{(\dot{x},.)}\alpha\big)(x)=0\big\}
$$

\begin{multline*}
\big(\textrm{resp.}\quad \big\{(\dot{\alpha},\dot{\beta},\dot{x})\in \mathbb{R} H^0(X;\mathcal{L}^d)^2\times T_x\mathbb{R} X\mid \\
\big(\dot{\alpha}\nabla\beta-\beta\nabla\dot{\alpha}
+\alpha\nabla\dot{\beta}-\dot{\beta}\nabla\alpha
+\alpha\nabla^2_{(\dot{x},.)}\beta-\beta\nabla^2_{(\dot{x},.)}\alpha\big)(x)=0\big\}.\big)
\end{multline*}
We remark that $\textrm{d}\pi_{H\mid(\alpha,\beta,x)}$ is (almost everywhere) an isometry, because on $\mathcal{I}$ we put the (singular) metric $\pi_H^*\langle\cdot ,\cdot\rangle_{L^2}$. For any $x\in X$ (resp. $x\in\mathbb{R} X$) we will compute  the normal Jacobian $(\pi_H^{-1})^*\textrm{Jac}_N(\pi_X)$ (resp. $(\pi_{\mathbb{R} H }^{-1})^*\textrm{Jac}_N(\pi_{\mathbb{R} X})$) at a point ${(\alpha,\beta)}\in \pi_H(\pi_X^{-1}(x))$ by using the following two linear maps:
  
  \begin{equation}\label{a} A_{\alpha\beta}:H^0(X;\mathcal{L}^d)\times H^0(X;\mathcal{L}^d)\rightarrow T_x^*X\otimes \mathcal{L}_x^{2d}
  \end{equation}
  
  $$\big(\textrm{resp.}\quad A_{\alpha\beta}:\mathbb{R} H^0(X;\mathcal{L}^d)\times \mathbb{R} H^0(X;\mathcal{L}^d)\rightarrow \mathbb{R}(T^*X\otimes \mathcal{L}^{2d})_x\big) $$  
and 
 \begin{equation}\label{b} B_{\alpha\beta}:T_x X\rightarrow T_x^*X\otimes \mathcal{L}_x^{2d} 
  \end{equation}
  $$\big(\textrm{resp.}\quad B_{\alpha\beta}:T_x\mathbb{R} X\rightarrow  
  \mathbb{R}(T^*X\otimes \mathcal{L}^{2d})_x\big)$$
 defined by $$A_{\alpha\beta}(\dot{\alpha},\dot{\beta})=
   \big(\dot{\alpha}\nabla\beta-\beta\nabla\dot{\alpha}
+\alpha\nabla\dot{\beta}-\dot{\beta}\nabla\alpha\big)(x)$$ and $$B_{\alpha\beta}(\dot{x})=
\big(\alpha\nabla^2_{(\dot{x},.)}\beta-\beta\nabla^2_{(\dot{x},.)}\alpha\big)(x)$$
On $T_x^*X\otimes \mathcal{L}_x^{2d}$ (resp. $\mathbb{R}(T^*X\otimes \mathcal{L}^{2d})_x$) we have the Hermitian (resp. scalar) product induced by $h$.

  \begin{prop}\label{normj} Following the notation of Sections \ref{lefincvar} and \ref{lefcoarea},  for any $x\in X$ (resp. $x\in\mathbb{R} X$) and any $(\alpha,\beta)\in \pi_{ H}(\pi^{-1}_{X}(x))$ (resp. $\pi_{\mathbb{R} H}(\pi^{-1}_{\mathbb{R} X}(x))$), we have $$\big((\pi_H^{-1})^*\textrm{Jac}_N\pi_X\big)(\alpha,\beta)=\frac{\textrm{Jac}_N(A_{\alpha\beta})}{\textrm{Jac}(B_{\alpha\beta})},$$ where $A_{\alpha\beta}$ and $B_{\alpha\beta}$ are the linear maps defined in (\ref{a}) and (\ref{b}).
  \end{prop}
  
\begin{proof} Recall that a vector $(\dot{\alpha},\dot{\beta},\dot{x})\in H^0(X;\mathcal{L}^d)^2\times T_xX$ is in the tangent space of $\mathcal{I}$ at $(\alpha,\beta,x)$ if and only if 
$$\big(\dot{\alpha}\nabla\beta-\beta\nabla\dot{\alpha}
+\alpha\nabla\dot{\beta}-\dot{\beta}\nabla\alpha
+\alpha\nabla^2_{(\dot{x},.)}\beta-\beta\nabla^2_{(\dot{x},.)}\alpha\big)
(x)=0.$$
In particular, the vector $\dot{x}$ is uniquely determined by \begin{equation}\label{vect}
\dot{x}=-\big(\alpha\nabla^2\beta-\beta\nabla^2\alpha\big)(x)^{-1}\circ \big(\dot{\alpha}\nabla\beta-\beta\nabla\dot{\alpha}
+\alpha\nabla\dot{\beta}-\dot{\beta}\nabla\alpha\big)(x).
\end{equation}
The map $\big((\pi_H^{-1})^*\textrm{d}\pi_X\big)(\alpha,\beta)$ sends $(\dot{\alpha},\dot{\beta})$ to $\dot{x}$ and then, by (\ref{vect}) and by the definition of $A_{\alpha\beta}$ and $B_{\alpha\beta}$, we have  $$\big((\pi_H^{-1})^*\textrm{d}\pi_X\big)(\alpha,\beta)=-B_{\alpha\beta}^{-1}\circ A_{\alpha\beta}.$$ 
Passing to the normal Jacobian and using that $\pi_H$ is a local isometry, we get $\big((\pi_H^{-1})^*\textrm{Jac}_N\pi_X\big)(\alpha,\beta)=\textrm{Jac}(B_{\alpha\beta})^{-1}\textrm{Jac}_N(A_{\alpha\beta}).$
\end{proof}
Fix  real holomorphic coordinates $(x_1,\dots,x_n)$ in a neighborhood of a point $x\in\mathbb{R} X$ such that $(\frac{\partial}{\partial x_1},\dots,\frac{\partial}{\partial x_n})$ is an orthonormal basis of $T_xX$ (resp. $T_x\mathbb{R} X$). We want to compute the integral \begin{equation}\label{codensity}\int_{\pi_{\mathbb{R} H}(\pi^{-1}_{\mathbb{R} X}(x))}\frac{1}{|(\pi_{\mathbb{R} H}^{-1})^*\textrm{Jac}_N(\pi_{\mathbb{R} X})|}\textrm{d}\mu_{\mid\pi_{\mathbb{R} H}(\pi^{-1}_{\mathbb{R} X}(x))}
  \end{equation}
that appears in (\ref{realcoarea}). 

For any  $(\alpha,\beta) \in H^0(X;\mathcal{L}^d)^2$ (resp. $\mathbb{R} H^0(X;\mathcal{L}^d)^2$) we have
$$\alpha=\sum_{i=0}^n a_i\sigma_i+\sum_{1\leqslant k\leqslant l\leqslant n}a_{kl}\sigma_{kl}+\tau$$
$$\beta=\sum_{i=0}^n b_i\sigma_i+\sum_{1\leqslant k\leqslant l\leqslant n}b_{kl}\sigma_{kl}+\tau'$$
where $\tau, \tau'\in \ker J^2_x$ and $\sigma_i,\sigma_{kl}$ are the peak section of Lemma \ref{realpeak}.\\
We remark that $(\alpha,\beta)\in\pi_H(\pi^{-1}_X(x))$ if and 
only if $a_0b_i-a_ib_0=0$ $\forall i=1,\dots,n$, and also that the definition of $\textrm{Jac}_N(\pi_{X})$ involves only the $2$-jets of sections. With this remark in mind we define the following spaces:
\begin{itemize}
\item $K_2\doteqdot(\ker J^2_x\times\ker J^2_x)\subset H^0(X;\mathcal{L}^d)^2$ (resp. $\mathbb{R} H^0(X;\mathcal{L}^d)^2$);
\item $H_2\doteqdot Vect\{(\sigma_i,0),(\sigma_{kl},0),(0,\sigma_i),(0,\sigma_{kl})\} \subset H^0(X:\mathcal{L}^d)^2$ (resp. $\mathbb{R} H^0(X;\mathcal{L}^d)^2$) for $i=0,\dots,n$ and $1\leq l\leq k\leq n$;
\item $Q=H_2\cap \pi_H(\pi^{-1}_X(x))$.
\end{itemize}
We see $Q$ as the product of the intersection of quadrics:
$$\tilde{Q}=\big\{(a_0,b_0,\dots,a_n,b_n)\in \mathbb{R}^{2(n+1)}\mid a_0b_i-a_ib_0=0\quad\forall i=1,\dots,n\big\}$$
with the vector space $\mathbb{R}^{n(n+1)}$ of coordinates $a_{ij}$ and $b_{ij}$ for $1\leq i\leq j\leq n$.\\
Let $\pi_2:K_2^{\perp}\rightarrow H_2$ be the orthogonal projection. A consequence of Proposition \ref{ortpeak} is that, for large $d$, the map $\pi_2$ is invertible. 
\\

\begin{prop}\label{comput} Following the notation of Section \ref{lefincvar} and \ref{lefcoarea}, let $A_{\alpha\beta}$ and $B_{\alpha\beta}$ be the linear applications defined in (\ref{a}) and (\ref{b}). Then, in the complex case, under the hypothesis of Theorem \ref{eqcom}, 
$$(\pi_2^{-1})_*\textrm{Jac}_N(A_{\alpha\beta})=\det\bigg(\pi\delta^2_{\mathcal{L}}d^{2n+1}
\big((a_i\bar{a}_j+b_i\bar{b}_j)E_{ij}
+(|a_0|^2+ |b_0|^2)Id+O(\frac{1}{\sqrt{d}})\big)_{ij}\bigg)$$
$$(\pi_2^{-1})_*\textrm{Jac}(B_{\alpha\beta})=\bigg|\det 
\bigg(\pi\delta_{\mathcal{L}}\sqrt{d}^{2n+2}\big((a_0b_{ij}-b_0a_{ij})\tilde{E}_{ij}
+O(\frac{1}{\sqrt{d}})\big)_{ij}\bigg)\bigg|^2$$
and, in the real case, under the hypothesis of the Theorem \ref{eqreal},

$$(\pi_2^{-1})_*\textrm{Jac}_N(A_{\alpha\beta})=
\sqrt{\det\bigg(\pi\delta_{\mathcal{L}}^2d^{2n+1}
\big((a_ia_j+b_ib_j)E_{ij}+(a_0^2+b_0^2)Id+O(\frac{1}{\sqrt{d}})\big)_{ij}\bigg)}$$
$$(\pi_2^{-1})_*\textrm{Jac}(B_{\alpha\beta})=\det
\bigg(\pi\delta_{\mathcal{L}}\sqrt{d}^{2n+2}\big((a_0b_{ij}-b_0a_{ij})\tilde{E}_{ij}
+O(\frac{1}{\sqrt{d}})\big)_{ij}\bigg)$$

where $\tilde{E}_{ij}$ for $1\leqslant i\leqslant j\leqslant n$ and $E_{ij}$ for $i,j=1,\dots,n$ are the matrices defined in Definition \ref{mat}. 
\end{prop}
\begin{proof}
Let $e$ be a local trivialization of $\mathcal{L}$ at $x$ as in Section \ref{lefpeak} and let  $(\sigma_i)_{i=0,\dots,n}$, 
  $(\sigma_{kl})_{1\leqslant k\leqslant l\leqslant n}$ be as in Lemma \ref{peak} (resp. Lemma \ref{realpeak}).
For any $(\alpha,\beta) \in H^0(X;\mathcal{L}^d)^2$ (resp. $\mathbb{R} H^0(X;\mathcal{L}^d)^2$) we have 
$$\alpha=\sum_{i=0}^n a_i\sigma_i+\sum_{1\leqslant k\leqslant l\leqslant n}a_{kl}\sigma_{kl}+\tau$$
$$\beta=\sum_{i=0}^n b_i\sigma_i+\sum_{1\leqslant k\leqslant l\leqslant n}b_{kl}\sigma_{kl}+\tau'$$
where $\tau, \tau'\in \ker J^2_x.$ In particular, we have
$$\alpha(x)=a_0\sigma_0(x),\qquad \beta(x)=b_0\sigma_0(x),$$ $$\nabla\alpha(x)=\sum_{i=0}^na_i\nabla\sigma_i(x),\qquad\nabla\beta(x)=\sum_{i=0}^nb_i\nabla\sigma_i(x),$$ $$\nabla^2\alpha(x)=\sum_{i=0}^na_i\nabla^2\sigma_i(x)+\sum_{k,l}
a_{kl}\nabla^2\sigma_{kl}(x),\quad\nabla^2\beta(x)=\sum_{i=0}^nb_i\nabla^2\sigma_i(x)+\sum_{k,l}
b_{kl}\nabla^2\sigma_{kl}(x).$$
As basis for $T_xX$ and $T^*_xX\otimes \mathcal{L}^{2d}_x$ (resp.  $T_x\mathbb{R} X$ and $\mathbb{R}(T^*_xX\otimes \mathcal{L}^{2d}_x)$) we choose $(\frac{\partial}{\partial x_1},\dots,\frac{\partial}{\partial x_n})$ and $(\textrm{d}x_1\otimes e^{2d},\dots,\textrm{d}x_n\otimes e^{2d})$ respectively.
We choose $(\sigma_i,0)$ and  $(0,\sigma_i)$, $i=0,\dots,n$, as
a basis of a complement of $\ker J_x^1\times \ker J^1_x$. 
Thanks to Lemma \ref{ortpeak}, this basis is asymptotically orthonormal for the $L^2$-Hermitian product $\langle\cdot,\cdot\rangle_{L^2}$. By definition it is an orthonormal basis for the scalar product $(\pi_2^{-1})_*\langle\cdot,\cdot\rangle_{L^2}$ restricted to $H_2$ . 
Then we obtain, using Lemma \ref{limpeak},
$$\left\langle A_{\alpha\beta}(\sigma_0,0),\textrm{d}x_j\otimes e^{2d}\right\rangle
=b_j\sqrt{\pi}\delta_{\mathcal{L}}\sqrt{d}^{2n+1}+O(\sqrt{d}^{2n});$$
$$\left\langle A_{\alpha\beta}(\sigma_i,0),\textrm{d}x_j\otimes e^{2d}\right\rangle=
-b_0\sqrt{\pi}\delta_{\mathcal{L}}\sqrt{d}^{2n+1}\delta_{ij}+O(\sqrt{d}^{2n}) \quad\textrm{for}\quad i=1,\dots,n;$$
$$\left\langle A_{\alpha\beta}(0,\sigma_0),\textrm{d}x_j\otimes e^{2d}\right\rangle=
-a_j\sqrt{\pi}\delta_{\mathcal{L}}\sqrt{d}^{2n+1}+O(\sqrt{d}^{2n});$$
$$\left\langle A_{\alpha\beta}(0,\sigma_i),\textrm{d}x_j\otimes e^{2d}\right\rangle=
a_0\sqrt{\pi}\delta_{\mathcal{L}}\sqrt{d}^{2n+1}\delta_{ij}+O(\sqrt{d}^{2n})
\quad\textrm{for}\quad i=1,\dots,n;$$
$$\left\langle B_{\alpha\beta}(\frac{\partial}{\partial x_i}),\textrm{d}x_j\otimes e^{2d}\right\rangle
=(a_0b_{ij}-b_0a_{ij})\pi\delta_{\mathcal{L}}\sqrt{d}^{2n+2}+O(\sqrt{d}^{2n+1}) \quad\textrm{for}\quad i\neq j;$$ 
$$\left\langle B_{\alpha\beta}(\frac{\partial}{\partial x_k}),\textrm{d}x_k\otimes e^{2d}\right\rangle
=\sqrt{2}(a_0b_{kk}-b_0a_{kk})\pi\delta_{\mathcal{L}}\sqrt{d}^{2n+2}+O(\sqrt{d}^{2n+1}).$$
where the Hermitian (resp. scalar) product $\left\langle\cdot,\cdot \right\rangle$  on $T^*_xX\otimes \mathcal{L}^{2d}_x$  (resp. $\mathbb{R}(T^*_xX\otimes \mathcal{L}^{2d}_x)$) is  induced by the Hermitian metric $h$ on $\mathcal{L}$.\\ 
What we have just computed are the coefficients of the matrices of $A_{\alpha\beta}$ and $B_{\alpha\beta}$ with respect to our choice of basis and with respect to the scalar product $(\pi_2^{-1})_*\langle\cdot,\cdot\rangle_{L^2}$. We recall that $B_{\alpha\beta}$ is a square matrix and that $\textrm{Jac}_N(A_{\alpha\beta})=\sqrt{\textrm{Jac}(A_{\alpha\beta}A_{\alpha\beta}^*)}$.\\ More precisely, as $d\rightarrow \infty$ , $A_{\alpha\beta}$ is equivalent to the following matrix: 

\[\sqrt{\pi}\delta_{\mathcal{L}}\sqrt{d}^{2n+1}
\begin{bmatrix}
b_1 & -b_0 &0 & \dots & 0 & -a_1 & a_0 & 0 &\dots& 0
 \\
b_2 & 0 & -b_0 & \dots & 0 & -a_2  & 0 & a_0 & \dots& 0
 \\
 \dots & \dots & \dots & \dots & \dots & \dots  & \dots & \dots& \dots& \dots
 \\
b_n & 0 & 0 &   \dots& -b_0  & -a_n &  0 &0 &\dots& a_0
\end{bmatrix}
\]
 and $B_{\alpha\beta}$ to the following one: 
\[\pi\delta_{\mathcal{L}}\sqrt{d}^{2n+2} 
\begin{bmatrix}
\sqrt{2}(a_0b_{11}-b_0a_{11}) & a_0b_{12}-b_0a_{12} & \dots &  a_0b_{1n}-b_0a_{1n}
 \\
a_0b_{21}-b_0a_{21} & \sqrt{2}(a_0b_{22}-b_0a_{22}) & \dots  &a_0b_{2n}-b_0a_{2n}
 \\
 \dots & \dots & \dots & \dots  
 \\

a_0b_{n1}-b_0a_{n1} & a_0b_{n2}-b_0a_{n2} & \dots & \sqrt{2}(a_0b_{nn}-b_0a_{nn})
\end{bmatrix}
\]

A direct computation shows us that $A_{\alpha\beta}A_{\alpha\beta}^*$ is the matrix $$(\pi\delta^2_{\mathcal{L}}d^{2n+1})
\big((a_i\bar{a}_j+b_i\bar{b}_j)E_{ij}
+(|a_0|^2+|b_0|^2)Id+O(\frac{1}{d})\big).$$ The results follows.
 \end{proof}

By Proposition \ref{comput} we have

$$(\pi_2^{-1})_*\frac{1}{\textrm{Jac}_N(\pi_{X})}=
 (\pi d)^n
\bigg(
\frac{\textrm{Jac}_{\mathbb{R}}\big((a_0b_{ij}-b_0a_{ij})\tilde{E}_{ij}\big)}{\det\big((a_i\bar{a}_j+b_i\bar{b}_j)E_{ij}+
(|a_0|^2+|b_0|^2)Id\big)}
+O(\frac{1}{\sqrt{d}})\bigg)$$

$$(\pi_2^{-1})_*\frac{1}{\textrm{Jac}_N(\pi_{\mathbb{R} X})}=
 \sqrt{\pi d}^n
 \bigg(
 \frac{\det\big((a_0b_{ij}-b_0a_{ij})\tilde{E}_{ij}\big)}{\sqrt{\det\big((a_ia_j+b_ib_j)E_{ij}+(a_0^2+b_0^2)Id\big)}}
+O(\frac{1}{\sqrt{d}})\bigg).$$
We want to integrate this quantity over $\pi_{\mathbb{R} H}\big(\pi^{-1}_{\mathbb{R} X}(x)\big)$.
We recall that the measure $\textrm{d}\mu_{\mid\pi_{\mathbb{R} H}\big(\pi^{-1}_{\mathbb{R} X}(x)\big)}$ is the following one: first we restrict the scalar product $\langle\cdot,\cdot\rangle_{L^2}$ on $\mathbb{R} H^0(X;\mathcal{L}^d)^2$ to $\pi_{\mathbb{R} H}(\pi^{-1}_{\mathbb{R} X}(x))$, that is a codimension $n$ submanifold, then  we  consider the Riemannian measure associated with this metric, and finally we multiply it by the  factor $\frac{1}{\pi^{N_d}}e^{-\norm{\alpha}_{L^2}^2-\norm{\beta}_{L^2}^2}$, where $N_d=\dim H^0(X;\mathcal{L}^d\mathcal{L}^d)$.
Then  (\ref{codensity}) is equal to

\begin{multline}\label{dens}
\int_{\pi_{\mathbb{R} H}\big(\pi^{-1}_{\mathbb{R} X}(x)\big)}\frac{|\textrm{Jac}(B_{\alpha\beta})|}{|\textrm{Jac}_N(A_{\alpha\beta})|}\textrm{d}\mu_{\mid \pi_{\mathbb{R} H}\big(\pi^{-1}_{\mathbb{R} X}(x)\big)}
=\int_{K_2^{\perp}\cap \pi_{\mathbb{R} H}\big(\pi^{-1}_{\mathbb{R} X}(x)\big)\oplus K_2}\frac{|\textrm{Jac}(B_{\alpha\beta})|}{| \textrm{Jac}_N(A_{\alpha\beta})|}\textrm{d}\mu_{\mid \pi_H\big(\pi^{-1}_{\mathbb{R} X}(x)\big)}=\\
=\int_{K_2^{\perp}\cap \pi_{\mathbb{R} H}\big(\pi^{-1}_{\mathbb{R} X}(x)\big)}\frac{|Jac(B_{\alpha\beta})|}{|\textrm{Jac}_N(A_{\alpha\beta})|}\textrm{d}\mu_{\mid K_2^{\perp}\cap \pi_{\mathbb{R} H}\big(\pi^{-1}_{\mathbb{R} X}(x)\big)}
=\int_Q(\pi_2^{-1})_*\frac{|\textrm{Jac}(B_{\alpha\beta})|}{|\textrm{Jac}_N(A_{\alpha\beta})|}(\pi_{2*}\textrm{d}\mu_{\mid K_2^{\perp}\cap \pi_{\mathbb{R} H}\big(\pi^{-1}_{\mathbb{R} X}(x)\big)}).
\end{multline}
By Proposition \ref{ortpeak}, the pushforward measure $(\pi_2)_*(\mu_{\mid K_2^{\perp}})$ on $H_2$ coincides with the Gaussian measure associated with the orthonormal basis $\{(\sigma_i,0),(\sigma_{kl},0),(0,\sigma_i),(0,\sigma_{kl})\}_{\substack{1\leq i\leq n \\ 1\leq k\leq l\leq n}}$ up to a $O(\frac{1}{\sqrt{d}})$ term. As a consequence we have that $(\pi_{2*}\textrm{d}\mu_{\mid K_2^{\perp}\cap \pi_{\mathbb{R} H}(\pi^{-1}_{\mathbb{R} X}(x))})$ is equal to $$\textrm{d}\mu_Q=\frac{e^{-\sum_ia_i^2-\sum_i
b_i^2-\sum_{i,j}a_{ij}^2-\sum_{i,j}b_{ij}^2}}{\pi^{n+1+\frac{n(n+1)}{2}}}\dV_Q$$ up to a $O(\frac{1}{\sqrt{d}})$ term, where $\dV_Q$ is the Riemannian volume form of  $Q$.
We then have that (\ref{dens}) is equal to 

\begin{multline}\label{quadric}
\int_Q(\pi_2^{-1})_*\frac{|\textrm{Jac}(B_{\alpha\beta})|}{|\textrm{Jac}_N(A_{\alpha\beta})|}\textrm{d}\mu_Q+O(\frac{1}{\sqrt{d}})=
\\=\int_{\begin{array}{l} a_i,b_i,a_{ij},b_{ij}\\ a_0b_i-b_0a_i=0\end{array}}\sqrt{\pi d}^n
 \big(\frac{\big|\det\big((a_0b_{ij}-b_0a_{ij})\tilde{E}_{ij}\big)\big|}{\sqrt{\det\big((a_ia_j+b_ib_j)E_{ij}+(a_0^2+b_0^2)Id\big)}}
\textrm{d}\mu_Q +O(\frac{1}{\sqrt{d}})\big).
\end{multline}
Putting (\ref{quadric}) in (\ref{realcoarea}) and using Proposition \ref{normj}, we obtain Proposition \ref{espform}.\qed
\subsection{Computation of the universal constant}\label{seccomput}
The purpose of this section is the explicit computation of the function 
$R_d(x)$ that appears in Proposition \ref{espform}. We use the notation of Section \ref{lefcomputation}.

To understand $R_d(x)$, we have to compute 
$$
\sqrt{\pi d}^n\int_{Q}\frac{\big| \det\big((a_0b_{ij}-b_0a_{ij})\tilde{E}_{ij}\big)\big|}{\sqrt{\det\big((a_ia_j+b_ib_j)E_{ij}+(a_0^2+b_0^2)Id\big)}}
\frac{e^{-\sum_ia_i^2-\sum_i
b_i^2-\sum_{i,j}a_{ij}^2-\sum_{i,j}b_{ij}^2}}{\pi^{n+1+\frac{n(n+1)}{2}}}\dV_Q.
$$
The main result of this section is the following computation:

\begin{prop}\label{unconst} Let $Q$ be as in Proposition \ref{espform}. Then $$\sqrt{\pi d}^n\int_Q\frac{\big| \det\big((a_0b_{ij}-b_0a_{ij})\tilde{E}_{ij}\big)\big|}{\sqrt{det\big((a_ia_j+b_ib_j)E_{ij}+(a_0^2+b_0^2)Id\big)}}
\frac{e^{-\sum_ia_i^2-\sum_i
b_i^2-\sum_{i,j}a_{ij}^2-\sum_{i,j}b_{ij}^2}}{\pi^{n+1+\frac{n(n+1)}{2}}}\dV_Q$$
 is equal to \begin{equation}\label{univ}\left\{\begin{array}{ll}
\frac{n!!}{(n-1)!!}e_{\mathbb{R}}(n)\frac{\pi}{2}\sqrt{d}^n &\textrm{if n is odd}
\\\frac{n!!}{(n-1)!!}e_{\mathbb{R}}(n)\sqrt{d}^n & \textrm{if n is even}.
\end{array}\right.
\end{equation}
where $\tilde{E}_{ij}$ and $E_{ij}$ are the matrices of Definition \ref{mat} and $e_{\mathbb{R}}(n)$ is the expected value of the determinant of (the absolute value of) the real symmetric matrices.
\end{prop}
We recall that $Q\subset\mathbb{R}^{2(n+1)+n(n+1)}$ is the product of the intersection of quadrics:
$$\tilde{Q}=\big\{(a_0,b_0,\dots,a_n,\dots,b_n)\in \mathbb{R}^{2(n+1)}\mid a_0b_i-a_ib_0=0\quad\forall i=1,\dots,n\big\}$$
with 
 $\mathbb{R}^{n(n+1)}$ of coordinates $a_{ij}$ and $b_{ij}$ for $1\leq i\leq j\leq n$.
We consider the parametrization 
$\psi:\mathbb{R}^{(n+2)}\rightarrow \tilde{Q}$ defined by $$\psi(a,b,t_1,\dots,t_n)=(a,b,at_1,bt_1,\dots,at_n,bt_n).$$
\begin{lemm} We have $\textrm{Jac} (\psi)=\sqrt{1+\sum_it_i^2}\sqrt{(a^2+b^2)}^n$.
\end{lemm}
\begin{proof}
A computation gives us
\[\textrm{Jac}\psi \textrm{Jac}\psi^t=\det
\begin{bmatrix}
1+\sum_{i=1}^nt_i^2 & 0 & t_1a & t_2a &  \dots& \dots & t_na
 \\
0 & 1+\sum_{i=1}^nt_i^2 & t_1b & t_2b & \dots  & \dots& t_nb
 \\
 t_1a & t_1b & a^2+b^2 & 0 & \dots & \dots & 0
 \\

t_2a & t_2b & 0 & a^2+b^2 & \dots  & \dots &  0
\\
\dots & \dots & \dots & \dots & \dots & \dots  & \dots
\\
\dots & \dots & \dots & \dots & \dots & \dots  & \dots
\\
t_na & t_nb & 0& 0 & \dots& \dots   &a^2+b^2  
\end{bmatrix}
\]
We develop the last line and we obtain
\[(a^2+b^2)\det
\begin{bmatrix}
1+\sum_{i=1}^nt_i^2 & 0 & t_1a & t_2a &  \dots& \dots & t_{n-1}a
 \\
0 & 1+\sum_{i=1}^nt_i^2 & t_1b & t_2b & \dots  & \dots& t_{n-1}b
 \\
 t_1a & t_1b & a^2+b^2 & 0 & \dots & \dots & 0
 \\

t_2a & t_2b & 0 & a^2+b^2 & \dots  & \dots &  0
\\
\dots & \dots & \dots & \dots & \dots & \dots  & \dots
\\
\dots & \dots & \dots & \dots & \dots & \dots  & \dots
\\
t_{n-1}a & t_{n-1}b & 0& 0 & \dots& \dots   &a^2+b^2  
\end{bmatrix}
\]
\[+(-1)^{n}t_nb\cdot\det
\begin{bmatrix}
1+\sum_{i=1}^nt_i^2  & t_1a & t_2a &  \dots& \dots & t_na
 \\
0  & t_1b & t_2b & \dots  & \dots& t_nb
 \\
 t_1a  & a^2+b^2 & 0 & \dots & \dots & 0
 \\

t_2a &  0 & a^2+b^2 & \dots  & \dots &  0
\\
\dots  & \dots & \dots & \dots & \dots  & \dots
\\
\dots  & \dots & \dots & \dots & \dots  & 0
\\
t_{n-1}a  & 0& 0 & \dots& a^2+b^2    &0 
\end{bmatrix}
\]
\[+(-1)^{n+1}t_na\cdot\det
\begin{bmatrix}
 0 & t_1a & t_2a &  \dots& \dots & t_na
 \\
 1+\sum_{i=1}^nt_i^2 & t_1b & t_2b & \dots  & \dots& t_nb
 \\
 t_1b & a^2+b^2 & 0 & \dots & \dots & 0
 \\

 t_2b & 0 & a^2+b^2 & \dots  & \dots &  0
\\
  \dots & \dots & \dots & \dots & \dots  & \dots
\\
  \dots & \dots & \dots & \dots & \dots  & 0
\\
  t_{n-1}b & 0& 0 & \dots& a^2+b^2   &0 
\end{bmatrix}
\]
 For the second matrix we have:

\[\det
\begin{bmatrix}
1+\sum_{i=1}^nt_i^2  & t_1a & t_2a &  \dots& \dots & t_na
 \\
0  & t_1b & t_2b & \dots  & \dots& t_nb
 \\
 t_1a  & a^2+b^2 & 0 & \dots & \dots & 0
 \\

t_2a &  0 & a^2+b^2 & \dots  & \dots &  0
\\
\dots  & \dots & \dots & \dots & \dots  & \dots
\\
\dots  & \dots & \dots & \dots & \dots  & 0
\\
t_{n-1}a  & 0& 0 & \dots& a^2+b^2    &0 
\end{bmatrix}
\]

\[=(1+\sum_{i=1}^nt_i^2)\det
\begin{bmatrix}

 t_1b & t_2b & \dots  & \dots&\dots& t_nb
 \\
  a^2+b^2 & 0 & \dots & \dots & 0&0
 \\

 0 & a^2+b^2 & \dots  & \dots &  0&0
\\
 \dots & \dots & \dots & \dots  & \dots &\dots
\\
 \dots & \dots & \dots & a^2+b^2  & 0&0
\\
   0& \dots & \dots &0& a^2+b^2    &0 
\end{bmatrix}
\]
$$=(-1)^{n+1}(1+\sum_{i=1}^nt_i^2)t_nb(a^2+b^2)^{n-1}.$$
where the first equality is obtained by developping the first column and remarking that, in the development, each time we clear the $i$-th line, the $(i-1)$-th column and the last column are linearly equivalent.
Similarly, 
\[\det
\begin{bmatrix}
 0 & t_1a & t_2a &  \dots& \dots & t_na
 \\
 1+\sum_{i=1}^nt_i^2 & t_1b & t_2b & \dots  & \dots& t_nb
 \\
 t_1b & a^2+b^2 & 0 & \dots & \dots & 0
 \\

 t_2b & 0 & a^2+b^2 & \dots  & \dots &  0
\\
  \dots & \dots & \dots & \dots & \dots  & \dots
\\
  \dots & \dots & \dots & \dots & \dots  & 0
\\
  t_{n-1}b & 0& 0 & \dots& a^2+b^2   &0 
\end{bmatrix}
\]
$$=(-1)^n(1+\sum_{i=1}^nt_i^2)t_na(a^2+b^2)^{n-1}.$$
Then we have
\[\textrm{Jac}\psi \textrm{Jac}\psi^t=(a^2+b^2)\det
\begin{bmatrix}
1+\sum_{i=1}^nt_i^2 & 0 & t_1a & t_2a &  \dots& \dots & t_{n-1}a
 \\
0 & 1+\sum_{i=1}^nt_i^2 & t_1b & t_2b & \dots  & \dots& t_{n-1}b
 \\
 t_1a & t_1b & a^2+b^2 & 0 & \dots & \dots & 0
 \\

t_2a & t_2b & 0 & a^2+b^2 & \dots  & \dots &  0
\\
\dots & \dots & \dots & \dots & \dots & \dots  & \dots
\\
\dots & \dots & \dots & \dots & \dots & \dots  & \dots
\\
t_{n-1}a & t_{n-1}b & 0& 0 & \dots& \dots   &a^2+b^2  
\end{bmatrix}
\]
$$-(1+\sum_{i=1}^{n}t_i^2)t_n^2(a^2+b^2)^n.$$
Continuing to develop in the same way, we obtain by induction
$$\textrm{Jac}\psi \textrm{Jac}\psi^t=(1+\sum_{i=1}^nt_i^2)^2(a^2+b^2)^n-(1+\sum_{i=1}^nt_i^2)
\sum_{i=1}^nt_i^2(a^2+b^2)^n=(1+\sum_{i=1}^nt_i^2)(a^2+b^2)^n.$$
Passing to the square root we obtain the result.
\end{proof}
\begin{oss} In the following we will not write the symbols $\tilde{E}_{ij}$ and $E_{ij}$ defined in Definition \ref{mat} in order to simplify the notation. 
\end{oss}
After this change of variables, we have:

$$
\sqrt{\pi d}^n\int_{Q}\frac{|\det(a_0b_{ij}-b_0a_{ij})|}{\sqrt{\det((a_ia_j+b_ib_j)+(a_0^2+b_0^2)Id)}}
\frac{e^{-\sum_ia_i^2-\sum_i
b_i^2-\sum_{i,j}a_{ij}^2-\sum_{i,j}b_{ij}^2}}{\pi^{n+1+\frac{n(n+1)}{2}}}\dV_Q=
$$
$$
=\sqrt{\pi d}^n\int_{a,b,a_{ij},b_{ij},t_i\in\mathbb{R}}\frac{\big|\det\big(ab_{ij}-ba_{ij}\big)\big|}{\sqrt{\det\bigg((a^2+b^2)\big((t_it_j)+Id\big)\bigg)}}\times
$$
$$\times\frac{e^{-(1+\sum_it_i^2)(a^2+b^2)-\sum_{i,j}(a_{ij}^2+ b_{ij}^2)}}{\pi^{n+1+\frac{n(n+1)}{2}}}\sqrt{(1+\sum_it_i^2})\sqrt{(a^2+b^2)}^n\textrm{d}a\textrm{d}b\textrm{d}a_{ij}\textrm{d}b_{ij}\textrm{d}t_i
$$
Now $\det\big((a^2+b^2)((t_it_j)_{ij}+Id)\big)=(1+\sum_it_i^2)(a^2+b^2)^n$ so we obtain

$$\sqrt{\pi d}^n\int_{a,b,a_{ij},b_{ij},t_i\in\mathbb{R}}|\det\big((ab_{ij}-ba_{ij})_{ij}\big)|
\frac{e^{-(1+\sum_it_i^2)(a^2+b^2)-\sum_{i,j}(a_{ij}^2+ b_{ij}^2)}}{\pi^{n+1+\frac{n(n+1)}{2}}}\textrm{d}a\textrm{d}b\textrm{d}a_{ij}\textrm{d}b_{ij}\textrm{d}t_i.$$
It is more practical to see $(a,b)$ as a complex number $c\in\mathbb{C}$ and
also $(a_{ij},b_{ij})$ as $e_{ij}\in\mathbb{C}$. With a slight abuse of notation, we denote $\textrm{d}c$ and $\textrm{d}e_{ij}$ instead of  $\frac{-1}{2i}\textrm{d}c\textrm{d}\bar{c}$ and $\frac{-1}{2i}\textrm{d}e_{ij}\textrm{d}\bar{e}_{ij}$. Then we have

$$\sqrt{\pi d}^n\int_{c\in\mathbb{C},e_{ij}\in\mathbb{C},t_i\in\mathbb{R}}\big|\det\big(\Im(\bar{c}e_{ij})\big)\big|
\frac{e^{-(1+\sum_it_i^2)| c |^2-\sum_{i,j}|e_{ij}|^2}}{\pi^{n+1+\frac{n(n+1)}{2}}}\textrm{d}c\textrm{d}e_{ij}\textrm{d}t_i.$$
Now, set $\tilde{c}=(\sqrt{1+\sum_it_i^2})c$ and then $\tilde{c}=re^{i\vartheta}$. We obtain

$$
\sqrt{\pi d}^n\int_{\tilde{c}\in\mathbb{C},e_{ij}\in\mathbb{C}}\big|\det\big(\Im(\bar{\tilde{c}}e_{ij})\big)\big|\frac{e^{-|\tilde{c}|^2-\sum_{i,j}|e_{ij}|^2}}{\pi^{1+\frac{n(n+1)}{2}}}\textrm{d}\tilde{c}\textrm{d}e_{ij}\cdot\int_{t_i\in\mathbb{R}}\frac{1}{\pi^{n}\sqrt{1+\sum_it_i^2}^{n+1}}
\textrm{d}t_i=
$$
\begin{equation*} \sqrt{\pi d}^n\int_{\vartheta\in (0,2\pi],e_{ij}\in\mathbb{C}}\big|\det\big(\Im(e^{-i\vartheta}e_{ij})\big)\big|\frac{e^{-\sum_{i,j}| e_{ij}|^2}}{\pi^{\frac{n(n+1)}{2}}}\textrm{d}e_{ij}\textrm{d}\vartheta
\cdot\int_{r=0}^{+\infty}\frac{r^{n+1}e^{-r^2}}{\pi}\textrm{d}r 
\cdot\int_{t_i\in\mathbb{R}}\frac{1}{\pi^{n}\sqrt{1+\sum_it_i^2}^{n+1}}
\textrm{d}t_i. 
\end{equation*}
Then, we  have to compute the three integrals appearing in the last equation.\\

For the first term, we have
$$\int_{\vartheta\in (0,2\pi],e_{ij}\in\mathbb{C}}\big|\det\big(\Im(e^{-i\vartheta}c_{ij})\big)\big|\frac{e^{-\sum_{i,j} |e_{ij}|^2}}{\pi^{\frac{n(n+1)}{2}}}\textrm{d}e_{ij}\textrm{d}\vartheta=$$
$$=2\pi\int_{e_{ij}\in\mathbb{C}}|\det(\Im e_{ij})|\frac{e^{-\sum_{i,j}|e_{ij}|^2}}{\pi^{\frac{n(n+1)}{2}}}\textrm{d}e_{ij}=$$
$$=2\pi\int_{b_{ij}\in\mathbb{R}}|\det(b_{ij})|\frac{e^{-\sum_{i,j}b_{ij}^2}}{\sqrt{\pi}^{\frac{n(n+1)}{2}}}\textrm{d}b_{ij}=2\pi e_{\mathbb{R}}(n).$$
Here, $e_{\mathbb{R}}(n)=\int_{B\in \Sym(n,\mathbb{R})}|\det B| \textrm{d}\mu_{\mathbb{R}}(B)$. For the explicit value of $e_{\mathbb{R}}(n)$, see \cite[Section 2]{gw2}.\\

For the second term, we consider the change of variable $r^2=\rho$ and we obtain  $$\int_{r=0}^{+\infty}\frac{r^{n+1}e^{-r^2}}{\pi}\textrm{d}r=
\frac{1}{2}\int_{\rho=0}^{+\infty}\frac{\rho^{\frac{n}{2}e^{-\rho}}}{\pi}\textrm{d}\rho=\frac{\Gamma(\frac{n}{2}+1)}{2\pi}$$
where $\Gamma$ is the Gamma function.\\

For the third term we use spherical coordinates and we obtain
$$\int_{t_i\in\mathbb{R}}\frac{1}{\sqrt{1+\sum_it_i^2}^{n+1}}
\textrm{d}t_i=\textrm{Vol}(S^{n-1})\int_{t=0}^{+\infty}\frac{t^{n-1}}{\sqrt{(1+t^2)}^{n+1}}\textrm{d}t.$$
where $$\textrm{Vol}(S^{n-1})=\frac{2\pi^{\frac{n}{2}}}{\Gamma(\frac{n}{2})}$$ is the volume of the $(n-1)$-dimensional sphere.
For 
$$\int_{t=0}^{+\infty}\frac{t^{n-1}}{(\sqrt{1+t^2})^{n+1}}\textrm{d}t=
\int_{t=0}^{+\infty}\sqrt{\frac{t^2}{(1+t^2)}}^{n-1}\frac{1}{1+t^2}\textrm{d}t$$
we make the change $\frac{t^2}{1+t^2}=1-u^2$ and we obtain
$$\int_0^1 \sqrt{1-u^2}^{n-2}\textrm{d}u.$$ 
Finally, set $u=\sin(\theta)$ and we have

$$\int_0^{\pi/2}\cos^{n-1}(\theta) \textrm{d}\theta.$$
The formula $$\int \cos^{n-1}(\theta)\textrm{d}u=\frac{\sin(\theta)\cos^{n-2}(\theta)}{n-1}+\frac{n-2}{n-1}\int \cos^{n-3}(\theta)\textrm{d}u$$
tells us that $\int_0^{\pi/2}\cos^{n-1}(\theta) \textrm{d}\theta$ is equal to 
$$\frac{(n-2)!!}{(n-1)!!}$$
if $n$ is even and it is equal to 

$$\frac{(n-2)!!}{(n-1)!!}\frac{\pi}{2}$$
if $n$ is odd.\\

Putting together all the values of these three integrals and using that $$\frac{\Gamma(\frac{n}{2}+1)}{\Gamma(\frac{n}{2})}=\frac{n}{2},$$ we obtain Proposition \ref{unconst}.\qed

\subsection{End of the proofs  of Theorems \ref{num}, \ref{eqreal} and  \ref{eqcom}}
We use the notations of Sections \ref{lefcoarea}, \ref{lefcomputation} and \ref{seccomput}.
By Proposition \ref{espform}, we have 
 \begin{multline}\label{fine}
 \mathbb{E}[\mathbb{R}\nu_{\alpha\beta}](\varphi)=\int_{\mathbb{R} X}\varphi(x)\big(\int_Q\frac{|\textrm{Jac}_N(A_{\alpha\beta})|}{| \textrm{Jac}(B_{\alpha\beta})|}\textrm{d}\mu(a_i,b_i,a_{kl},b_{kl})+O(\frac{1}{d})\big)\dV_h=
\\=\int_{\mathbb{R} X}\varphi(x)\bigg(\int_{\begin{array}{l} a_i,b_i,a_{ij},b_{ij}\\ a_0b_i-b_0a_i=0\end{array}}\sqrt{\pi d}^n
 \big(
 \frac{\det(a_0b_{ij}-b_0a_{ij})}{\sqrt{\det\big((a_ia_j+b_ib_j)+(a_0^2+b_0^2)Id\big)}}
+O(\frac{1}{\sqrt{d}})\big)\textrm{d}\mu_Q \bigg)\dV_h
\end{multline} 
where  $\textrm{d}\mu_Q=
\frac{e^{-\sum_ia_i^2-\sum_i
b_i^2-\sum_{i,j}a_{ij}^2-\sum_{i,j}b_{ij}^2}}{\pi^{n+1+\frac{n(n+1)}{2}}}\dV_Q.$ By Proposition \ref{unconst}  we have that the inner term of the last equation
$$\sqrt{\pi d}^n\int_{\begin{array}{l} a_i,b_i,a_{ij},b_{ij}\\ a_0b_i-b_0a_i=0\end{array}}\frac{|\det(a_0b_{ij}-b_0a_{ij})|}{\sqrt{\det\big((a_ia_j+b_ib_j)+(a_0^2+b_0^2)Id\big)}}
\frac{e^{-\sum_ia_i^2-\sum_i
b_i^2-\sum_{i,j}a_{ij}^2-\sum_{i,j}b_{ij}^2}}{\pi^{n+1+\frac{n(n+1)}{2}}}\dV_Q$$
 is equal to $$\left\{\begin{array}{ll}
\frac{n!!}{(n-1)!!}e_{\mathbb{R}}(n)\frac{\pi}{2}\sqrt{d}^n &\textrm{if n is odd}
\\\frac{n!!}{(n-1)!!}e_{\mathbb{R}}(n)\sqrt{d}^n & \textrm{if n is even}.
\end{array}\right.$$
Theorem \ref{eqreal} is then obtained by dividing Equation (\ref{fine}) by $\sqrt{d}^n$ and then by passing to the limit.  \\
Theorem \ref{num} is  Theorem \ref{eqreal} for $\varphi=1$.\\
The proof of Theorem \ref{eqcom} follows the lines the proof of Theorem \ref{eqreal}. For the computation of the universal constant in this case, we  put $\varphi=1$ and we use Proposition \ref{asymlef}.\qed


\bibliographystyle{plain}
\bibliography{biblio}

\end{document}